\documentclass{article}
\usepackage[latin1]{inputenc}
\usepackage{amsmath,amsthm,amssymb}
\usepackage{amsfonts}
\usepackage{amsmath,amsthm,amssymb,amscd}
\usepackage{latexsym}
\usepackage{color}
\usepackage{graphicx}
\usepackage{mathrsfs}

\makeatletter \@addtoreset{equation}{section} \makeatother

\newtheorem{theorem}{Theorem}[section]
\newtheorem{definition}{Definition}[section]
\newtheorem{proposition}{Proposition}[section]
\newtheorem{lemma}{Lemma}[section]
\newtheorem{remark}{Remark}[section]

\newtheorem{corollary}[theorem]{Corollary}

\begin{document}
\title{Existence and multiplicity of solutions for fractional Schr\"{o}dinger-Kirchhoff equations with  Trudinger-Moser nonlinearity}
\author{Mingqi Xiang$^{a}$,
\ Binlin Zhang$^{b, }$\footnote{Corresponding author.
E-mail addresses:\,  xiangmingqi\_hit@163.com (M. Xiang), zhangbinlin2012@163.com (B. Zhang), 
dusan.repovs@guest.arnes.si (D. Repov\v{s})}\ \ 
and Du\v{s}an Repov\v{s}$^{c}$ \\ 
 \vspace{-0.1cm}
\footnotesize $^a$ College of Science, Civil Aviation University of China, Tianjin,
300300, P.R. China\\
 \vspace{-0.1cm}
\footnotesize $^b$ Department of Mathematics, Heilongjiang Institute of Technology, Harbin,
150050, P.R. China\\
\footnotesize $^c$ Faculty of Education and Faculty of Mathematics and Physics,
University of Ljubljana, 1000,
Slovenia}
\date{} 
\maketitle 
\vspace{-1cm}
\begin{abstract}
We study the existence and multiplicity of  solutions for a class of fractional Schr\"{o}dinger--Kirchhoff type equations with the Trudinger--Moser nonlinearity. More precisely, we consider
\begin{gather*}
\begin{cases}
M\big(\|u\|^{N/s}\big)\left[(-\Delta)^s_{N/s}u+V(x)|u|^{\frac{N}{s}-1}u\right]= f(x,u)
+\lambda  h(x)|u|^{p-2}u \ &{\rm in}\ \mathbb{R}^N,\\
\|u\|=\left(\iint_{\mathbb{R}^{2N}}\frac{|u(x)-u(y)|^{N/s}}{|x-y|^{2N}}dxdy+\int_{\mathbb{R}^N}V(x)|u|^{N/s}dx\right)^{s/N},
\end{cases}\end{gather*}
where $M:[0,\infty]\rightarrow [0,\infty)$ is a continuous function, $s\in (0,1)$, $N\geq2$, $\lambda>0$ is a parameter, $1<p<\infty$, $(-\Delta )^s_{N/s}$ is the fractional $N/s$--Laplacian,  $V:\mathbb{R}^N\rightarrow(0,\infty)$ is a continuous function, $f:\mathbb{R}^N\times\mathbb{R}\rightarrow\mathbb{R} $ is a continuous function, and $h:\mathbb{R}^N\rightarrow[0,\infty)$ is a measurable function.
First, using the mountain pass theorem, a nonnegative solution
is obtained when $f$ satisfies exponential growth conditions and $\lambda$
is large enough, and we prove that the solution converges to zero in $W_V^{s,N/s}(\mathbb{R}^N)$ as $\lambda\rightarrow\infty$. Then, using the Ekeland variational principle, a
nonnegative nontrivial solution is obtained when $\lambda$ is small enough, and we show that the solution converges to zero in $W_V^{s,N/s}(\mathbb{R}^N)$ as $\lambda\rightarrow0$. Furthermore, using the genus theory, infinitely many solutions are obtained
when $M$ is a special function and $\lambda$ is small enough.
We note that our paper covers a novel feature of Kirchhoff problems, that is,  the Kirchhoff function $M(0)=0$.\medskip

\medskip

\end{abstract}

\section{Introduction}

Given $s\in(0,1)$ and $N\geq2$, we study the following fractional Schr\"{o}dinger--Kirchhoff type equation:
\begin{equation}\label{eq1}
M\left(\|u\|^{N/s}\right)\left[(-\Delta)^s_{N/s}u+V(x)|u|^{N/s-1}u\right]= f(x,u)
+\lambda  h(x)|u|^{p-2}u\,\, \ \ {\rm in}\ \mathbb{R}^N,
\end{equation}
where
\begin{equation}\label{norm}
\|u\|=\left([u]_{s,N/s}^{N/s}
+\int_{\mathbb{R}^N}V(x)|u|^{N/s}dx\right)^{s/N},\quad
[u]_{s,N/s}=\left(\iint_{\mathbb{R}^{2N}}\frac{|u(x)-u(y)|^{N/s}}{|x-y|^{2N}}dxdy\right)^{{s}/{N}},
\end{equation}
$M:[0,\infty)\rightarrow[0,\infty)$ is a continuous function, $V:\mathbb{R}^N\rightarrow\mathbb{R}^+$ is a scalar
potential, $1<p<\infty$,  $h:\mathbb{R}^N\rightarrow[0,\infty)$ is a weight function,
$f:\mathbb{R}^N\times\mathbb{R}\rightarrow\mathbb{R}$ is a continuous function,
and $(-\Delta )^s_{N/s}$ is the associated fractional $N/s$-Laplace operator which, up to a normalization constant, is defined as
\begin{equation*}
(-\Delta)_{N/s}^s\varphi(x)=2
\lim_{\varepsilon\rightarrow 0^+}\int_{\mathbb{R}^N\setminus B_\varepsilon(x)}\frac{|\varphi(x)-\varphi(y)|^{N/s-2}(\varphi(x)-\varphi(y))}{|x-y|^{2N}}\,dy,\quad x\in\mathbb{R}^{N},
\end{equation*}
on functions $\varphi\in C_0^\infty(\mathbb{R}^N)$.
Hereafter, $B_\varepsilon(x)$ denotes the ball of
$\mathbb{R}^N$ centered at $x\in\mathbb{R}^N$ and
with radius $\varepsilon>0$.

 Equations of the type \eqref{eq1} are important in many fields of
science, notably in continuum mechanics, phase transition phenomena, population dynamics,
minimal surfaces, and anomalous diffusion, since they are a typical outcome of stochastic stabilization of L\'{e}vy processes, see \cite{r2, r3,r-2} and the references therein.
Moreover, such equations and the associated fractional operators allow us to develop a generalization of quantum mechanics and also
to describe the motion of a chain or an array of particles
which are connected by elastic springs, as well as unusual diffusion
processes in turbulent fluid motions and material transports in
fractured media, for more details see
\cite{r3,r7} and the references therein.
Indeed, the nonlocal fractional operators have been extensively studied by several authors in many different cases: bounded and unbounded domains, different behavior of the nonlinearity, etc. In particular, many
works focus on the subcritical and critical growth of the nonlinearity which allows us to treat the problem variationally by using general critical point theory.

This paper was motivated by some works which have  appeared in recent years.
On the one hand, the following nonlinear Schr\"{o}dinger equation
\begin{align}\label{eq1.01}
(-\Delta)^s u+V(x)u= f(x,u)\quad\mbox{in }\mathbb{R}^N,
\end{align}
was elaborated on by Laskin \cite{r-2} in the framework of quantum mechanics.
Equations of type \eqref{eq1.01}
have been extensively studied, see e.g.
\cite{Chang, MBR,MBRS,PXZ}. To the best of our knowledge, most of the works on fractional Laplacian problems involve the nonlinear terms satisfying polynomial growth, there are only few papers dealing with nonlinear term with exponential growth.

In recent years, some authors have paid considerable attention to the limiting case of the fractional Sobolev embedding, commonly known as the Trudinger-Moser case. In fact, when $N=2$, then $W^{1,2}(\Omega)\hookrightarrow L^r(\Omega)$ for $1\leq r<\infty$,
    but we cannot take $r=\infty$ for such an embedding. To fill this gap, for bounded domains $\Omega$, Trudinger \cite{Tr} proved that
that there exists $\tau>0$ such that $W^{1,2}_0(\Omega)$  is embedded into the Orlicz space $L_{\phi_\tau}(\Omega)$,
 determined by the Young function $\phi_\tau=\mbox{exp}({\tau t^2}-1)$. Afterwards,
 Moser found in \cite{Mo} the best exponent $\tau$ and in particular,
he obtained a result which  is now referred to as the Trudinger-Moser inequality.
For more details about Trudinger-Moser inequalities, we also refer to \cite{TO}.
Next, let us recall some useful results about fractional Trudinger-Moser inequality. Let $\omega_{N-1}$ be the surface area of the unit sphere in $\mathbb{R}^N$ and let $\Omega\subset\mathbb{R}^N$ be a bounded domain. Define $W_0^{s,N/s}(\Omega)$ as the completion of $C_0^\infty(\Omega)$ with respect to the norm $[\cdot]_{s,N/s}$.
In \cite{Martin}, Martinazzi proved that
 there exist positive constants $$\alpha_{N,s}
 =\frac{N}{\omega_{N-1}}\left(\frac{\Gamma((N-s)/2)}{\Gamma(s/2)2^s\pi^{N/2}}\right)^{-\frac{N}{N-s}}
 $$ and $C_{N,s}$ depending only on $N$ and $s$ such that
 \begin{align}\label{TMB}
 \sup_{\substack{u\in W_0^{s,N/s}(\Omega)\\
 [u]_{s,N/s}\leq1}}\int_\Omega
 \exp(\alpha|u|^{\frac{N}{N-s}})dx\leq C_{N,s}|\Omega|,
 \end{align}
 for all $\alpha\in [0,\alpha_{N,s}]$ and there exists $\alpha_{N,s}^*\geq \alpha_{N,s}$ such that the supremum in \eqref{TMB} is $\infty$ for $\alpha>\alpha_{N,s}^*$. However, it remains unknown whether
$
\alpha_{N,s}=\alpha_{N,s}^*.
$
Kozono {\it et al.} in \cite{KSW} proved that for all $\alpha>0$ and $u\in W^{s,N/s}(\mathbb{R}^N)$,
\begin{align*}
\int_{\mathbb{R}^N}\Phi_\alpha(u)dx<\infty,
\end{align*}
where
\begin{align}\label{ex}
\Phi_\alpha(t)=\exp\left(\alpha|t|^{\frac{N}{N-s}}\right)-\sum_{\substack{0\leq j<N/s-1\\ j\in\mathbb{N}}}
\frac{\alpha^j}{j!}|t|^{\frac{jN}{N-s}}.
\end{align}
Moreover, there exist positive constants $\alpha_{N,s}$ and $C_{N,s}$ depending only on $N$ and $s$ such that
\begin{align}\label{TM}
\int_{\mathbb{R}^N}\Phi_\alpha(u)dx\leq C_{N,s},\ \ \forall\alpha\in(0,\alpha_{N,s}),
\end{align}
for all $u\in W^{s,N/s}(\mathbb{R}^N)$ with $[u]_{s,N/s}\leq 1$.

 In the setting of the fractional Laplacian, Iannizzotto and Squassina in \cite{IM}
investigated existence of solutions for the following Dirichlet problem
\begin{align}\label{eq1.02}
\begin{cases}
(-\Delta)^{1/2} u=f(u)\ \ &{\rm in}\ (0,1),\\
u=0\quad \ \ &{\rm in}\  \mathbb{R}\setminus(0,1),
\end{cases}
\end{align}
where $(-\Delta)^{1/2}$ is the fractional Laplacian and $f(u)$ behaves like $\exp(\alpha |u|^2)$ as $u\rightarrow\infty$. Using the mountain pass theorem, they proved the existence of solutions for problem \eqref{eq1.02}. Subsequently,
 Giacomoni, Mishra and Sreenadh in \cite{GMS} studied the multiplicity of solutions for problems like \eqref{eq1.02} by using the Nehari manifold method. For more recent results on problem \eqref{eq1.02} in the higher dimensional case, we refer the interested reader to \cite{PS1}  and the references therein.
For the general fractional $p$--Laplacian in unbounded domains, Souza in \cite{Souza} considered the following nonhomogeneous fractional $p$--Laplacian equation
\begin{align}\label{eq1.03}
(-\Delta)_p^su+V(x)|u|^{p-2}u=f(x,u)+\lambda h\  \ {\rm in}\ \mathbb{R}^N,
\end{align}
where $(-\Delta)_p^s$ is the fractional $p$--Laplacian and the nonlinear term $f$ satisfies exponential growth.  He obtained
a nontrivial weak solution of the equation \eqref{eq1.03} by using fixed point theory. Li and Yang in \cite{LY} studied the following equation
\begin{align*}
(-\Delta)_p^\zeta u+V(x)|u|^{p-2}u=\lambda A(x)
|u|^{q-2}u+f(u)\ \ {\rm in} \ \mathbb{R}^N,
\end{align*}
where $p\geq 2$, $0<\zeta<1$, $1<q<p$, $\lambda>0$ is a real parameter,
$A$ is a positive function in $L^{\frac{p}{p-q}}(\mathbb{R}^N)$, $(-\Delta)_p^\zeta$ is the fractional $p$--Laplacian, and $f$ has exponential growth.

On the  other hand, Li and Yang  in \cite{LY2} studied the following Schr\"{o}dinger-Kirchhoff type
equation
\begin{align}\label{eq1.04}
\left(\int_{\mathbb{R}^N}(|\nabla u|^N+V(x)|u|^N)dx\right)^k(-\Delta_N u
+V(x)|u|^{N-2}u)=\lambda A(x)|u|^{p-2}u+ f(u)\quad\mbox{in }\mathbb{R}^N,
\end{align}
where $\Delta_N u={\rm div}(|\nabla u|^{N-2}\nabla u)$ is the $N$--Laplaician, $k>0$,
$V:\mathbb{R}^N\rightarrow(0,\infty)$ is continuous, $\lambda>0$ is a real parameter,
$A$ is a positive function in $L^{\frac{p}{p-q}}(\mathbb{R}^N)$ and $f$ satisfies exponential growth. By using the mountain pass theorem and Ekeland's variational principle, they obtained
two nontrivial solutions of \eqref{eq1.04} for the parameter $\lambda$ small enough.
Actually, the Kirchhoff--type problems,
which arise in various models of physical and biological systems,
have received a lot of attention in recent years. More precisely,
Kirchhoff established a model governed by the equation
\begin{align}\label{kirchhoff}
\rho\frac{\partial ^2u}{\partial t^2}-\left(
\frac{\rho}{h}+\frac{E}{2L}\int_0^L\left|\frac{\partial u}{\partial x}\right|^2dx\right)
\frac{\partial ^2u}{\partial x^2}=0,
\end{align}
  for all $x\in(0,L),t\geq0$, where $u=u(x,t)$ is the lateral displacement at the coordinate $x$ and the time $t$, $E$ is the Young modulus, $\rho$ is the mass density, $h$ is the cross-section area, $L$ is the length, and
  $\rho_0$ is the initial axial tension. Equation~\eqref{kirchhoff} extends the classical D'Alembert
wave equation by considering the effects of the changes in the length of the
strings during the vibrations.
Recently,  Fiscella and Valdinoci have proposed in \cite{FV} a stationary
Kirchhoff model driven by the fractional Laplacian
by taking into account the nonlocal aspect of the tension,
see \cite[Appendix A]{FV} for more details.
It is worth mentioning that when $s\rightarrow1^{-}$ and $M\equiv1$, problem \eqref{eq1} becomes
\begin{align*}
-\Delta_N u+V(x)|u|^{N-2}u=f(x,u)+\lambda h(x)|u|^{p-2}u,
\end{align*}
which was studied by many authors using variational methods, see for example
\cite{ADF,doMS,GPS, LLu}.

Inspired by the above works, we study in the present paper the existence, multiplicity and
asymptotic behavior of solutions of~\eqref{eq1} and
overcome the lack of compactness due to the presence
of exponential growth terms as well as the degenerate nature
of the Kirchhoff coefficient.
To the best of our knowledge, there are no results for \eqref{eq1}
of such  generality.

Throughout the paper, without explicit mention, we  assume
validity of conditions $(V_1)$, $(V_2)$ and $(\mathcal M)$ below:\smallskip

\noindent $(V_1)$
{\em$V:\mathbb{R}^N\rightarrow\mathbb{R}^+$ is a continuous
function and there exists $V_0>0$ such that
$\displaystyle{\inf_{\mathbb{R}^N}} V(x) \geq V_0$.}
\begin{itemize}
\item[$(V_2)$] {\em There exists $h>0$  such that
\begin{align*}
\lim_{|y|\rightarrow\infty} {\rm {meas}}\left\{x\in B_h(y):V(x)\leq c\right\}=0
\end{align*}
for all $c>0$.}
\end{itemize}
\noindent $(\mathcal M)$
{\em$M:\mathbb{R}^+_0\rightarrow \mathbb{R}^+_0$
is a continuous function satisfying the following properties:}
\begin{itemize}
 \item[$(M_1)$] {\em For any $\tau>0$ there exists $\kappa=\kappa(\tau)>0$
 such that $M(t)\geq\kappa$ for all $t\geq\tau$}.
  \item[$(M_2)$] {\em There exists $\theta\geq1$
  such that $tM(t)\leq \theta\mathscr{M}(t)$ for all $t\in \mathbb{R}^+_0$,
  where $\mathscr{M}(t)=\int_0^t M(\tau)d\tau$}.
\end{itemize}

Note that condition $(V_2)$, which is weaker than the
coercivity assumption $V(x)\rightarrow \infty$ as $|x|\rightarrow \infty$,
was first proposed by  Bartsch and  Wang in~\cite{BW}
to overcome the lack of compactness. The condition $(M_1)$ that means $M(t)>0$ for all $t>0$,
was originally used to get the multiplicity of solutions for a class of higher
order $p(x)$--Kirchhoff equations, see \cite{colpuc} for more details.

A typical example of $M$ is given by $M(t)=a+b\theta\,t^{\theta-1}$
for $t\geq0$, where $a,b\geq0$
and $a+b>0$. When $M$ is of this type, problem \eqref{eq1}
is said to be  degenerate if $a=0$, while it is
called non--degenerate if $a>0$.
Recently, fractional Kirchhoff problems have received more and more attention.
Some new existence results of solutions for fractional
non--degenerate Kirchhoff problems are given, for example, in
\cite{PP, PS, PXZ, XZF1}. For some recent results concerning the degenerate
case of Kirchhoff--type problems, we refer to \cite{AFP, capu, XMTZ, PXZ2, XZF2, XZR}
and the
references therein.
We stress that
the degenerate case is quite interesting and is treated in well--known
papers on Kirchhoff theory, see for example \cite{dAS}.
In the vast literature on degenerate Kirchhoff problems,
the transverse oscillations of a stretched string, with nonlocal
flexural rigidity, depends continuously on the Sobolev deflection
norm of $u$ via $M(\|u\|^2)$. From a physical point of view, the fact
that $M(0)=0$ means  that the base tension of the string is zero, which is a
very realistic model. Clearly, assumptions $(M_1)$--$(M_2)$ cover the degenerate case.

The natural solution space for \eqref{eq1} is
$W_V^{s,N/s}(\mathbb{R}^N)$, that is, the completion of
$C_0^\infty(\mathbb{R}^N)$ with respect to the norm $\|\cdot\|$
introduced in~\eqref{norm}. By \cite{PXZ}, we know that $(W_V^{s,N/s}(\mathbb{R}^N),\|\cdot\|)$ is a reflexive Banach space.
 Furthermore, for all $N/s\leq q<\infty$, the following embeddings
 \begin{align*}
 W_V^{s,N/s}(\mathbb{R}^N)\hookrightarrow
 W^{s,N/s}(\mathbb{R}^N)\hookrightarrow L^q(\mathbb{R}^N)
 \end{align*}
 are continuous, see \cite{r28}. Define
 \begin{align*}
 \lambda_1=\inf\left\{\frac{\|u\|^{N/s}}{\|u\|_{L^{\theta N/s}(\mathbb{R}^N)}^{\theta N/s}}:
 u\in W_V^{s,N/s}(\mathbb{R}^N)\setminus\{0\}\right\}.
 \end{align*}
 Clearly, $\lambda_1>0$.

Throughout the paper we assume that the nonlinear term
{\em $f:\mathbb{R}^N\times\mathbb{R}^+\rightarrow\mathbb{R}$
is a continuous function, with $f(x,t)\equiv 0$ for $t\leq0$
and $x\in\mathbb{R}^N$.}  In the following, we also require the following assumptions $(f_1)$--$(f_3)$:
\begin{itemize}
\item[$(f_1)$] {\it There exists $b_1,b_2>0$ and $0<\alpha_0<\alpha_{N,s}$, such that}
\begin{align*}
|f(x,t)|\leq b_1 t^{\theta N/s-1}+b_2\Phi_{\alpha_0}(t)\quad\mbox{\em for all }
(x,t)\in\mathbb{R}^N\times\mathbb{R}^+.
\end{align*}
where $\Phi_{\alpha}(t)$ is given in \eqref{ex}.
\item[$(f_2)$] {\em There exists  $\mu>\theta N/s$ such that
$$0<\mu F(x,t)\leq f(x,t)t,\quad
F(x,t)=\int_0^tf(x,\tau)d\tau,$$
whenever $x\in\mathbb{R}^N$ and $t\in \mathbb{R}^+$.}
\item[$(f_3)$] {\it The following holds:}
\begin{align*}
\limsup_{t\rightarrow0^+}\frac{F(x,t)}{t^{\theta N/s}}<\frac{s\mathscr{M}(1)}{N}\lambda_1,\ \
{\rm uniformly\ in}\ x\in\mathbb{R}^N.
\end{align*}
\end{itemize}
Note that $(f_3)$ is compatible with the condition $(M_2)$.
A typical example of $f$, satisfying
$(f_1)$--$(f_2)$, is given by $f(x,t)=\Phi_{\alpha_0}(t)+\mathcal{C}_0t^{\theta N/s-1}$, where $\mathcal{C}_0$ is a positive constant.

\medskip

We say that $u\in W_V^{s,N/s}(\mathbb{R}^N)$ is a (weak)
{\em solution} of problem \eqref{eq1}, if
\begin{gather*}
M(\|u\|^{N/s})\left(\langle u,\varphi\rangle_{s,N/s}+\int_{\mathbb{R}^N}V|u|^{N/s-2}u\varphi dx\right)
=\int_{\mathbb{R}^N}(f(x,u)+ \lambda h(x)|u|^{p-2}u)\varphi dx,\\ \langle u,\varphi\rangle_{s,N/s}=\iint_{\mathbb{R}^{2N}}\frac{\big[|u(x)-u(y)|^{N/s-2}(u(x)-u(y))\big]
\cdot\big[\varphi(x)-\varphi(y)\big]}{|x-y|^{2N}} dxdy,
\end{gather*}
for all $\varphi\in W_V^{s,N/s}(\mathbb{R}^N)$.
\medskip

First of all, for the case $N\theta/s<p<\infty$, by using the mountain pass theorem we can obtain the first existence result as follows.

\begin{theorem}\label{th1}
Assume that $V$ satisfies $(V_1)$--$(V_2)$,  $f$ satisfies $(f_1)$--$(f_3)$
and $M$ fulfills $(M_1)$--$(M_2)$. If $0\leq h\in L^\infty(\mathbb{R}^N)$ and $N\theta/s<p<\infty$,
then there exists
$\lambda^*>0$ such that for all $\lambda>\lambda^*$, problem \eqref{eq1}
admits a nontrivial nonnegative mountain pass solution $u_\lambda $ in
$W_V^{s,N/s}(\mathbb{R}^N)$. Moreover, $\lim_{\lambda\rightarrow\infty}\|u_\lambda\|=0$.
\end{theorem}

Then, for the case $1<p<N/s$, by utilizing the Ekeland variational principle we can get the second existence result as follows.
\begin{theorem}\label{th2}
Assume that $V$ satisfies $(V_1)$--$(V_2)$, $f$ satisfies $(f_1)$--$(f_3)$
and $M$ fulfills $(M_1)$--$(M_2)$. If $1<p<N/s$ and $0\leq h\in L^{\frac{N}{N-sp}}(\mathbb{R}^N)$,
then there exists
$\lambda_*>0$ such that for all $0<\lambda<\lambda_*$, problem \eqref{eq1}
admits a nontrivial nonnegative least energy solution
in $ W_V^{s,N/s}(\mathbb{R}^N)$. Moreover, $\lim_{\lambda\rightarrow0}\|u_\lambda\|=0$.
\end{theorem}

Finally, to study the existence of infinitely many solutions for problem \eqref{eq1} in the case $1<p<N/s$,
inspired by the method adopted in \cite{XMTZ}, we appeal to the genus theory.
However, we encounter some technical difficulties under the general assumptions $(M_1)$--$(M_2)$.  Therefore, we consider the classical Kirchhoff function, that is, $M(t)=a+b\theta t^{\theta-1}$ for all $t\geq0$, where $a\geq0,b\geq0$, $a+b>0$ and $\theta>1$.  As a consequence, we are able to prove a further result compared to Theorem \ref{th2}.

\begin{theorem}\label{th3}
Assume that $V$ satisfies $(V_1)$--$(V_2)$, $f$ satisfies $(f_1)$--$(f_3)$, and $M(t)=a+b\theta t^{\theta-1}$ for all $t\geq0$, where $a\geq0,b\geq0$, $a+b>0$ and $\theta>1$. If $1<p<N/s$ and $0\leq h\in L^{\frac{N}{N-sp}}(\mathbb{R}^N)$,
then there exists
$\lambda_{**}\in(0, \lambda_*]$ such that for all $0<\lambda<\lambda_{**}$, problem \eqref{eq1}
has infinitely many solutions in $ W_V^{s,N/s}(\mathbb{R}^N)$.
\end{theorem}

Here we point out that it remains  open to establish whether $\lambda_*=\lambda_{**}$
from Theorems \ref{th2} and \ref{th3}.
Moreover, it would be interesting to investigate whether there are solutions to problem \eqref{eq1} as $\lambda\in [\lambda_*, \lambda^*]$
from Theorems \ref{th1} and  \ref{th2}.

Let us simply describe the approaches to prove Theorems 1.1--1.3. To show the existence of at least one nonnegative solution of problem \eqref{eq1}, we shall use the mountain pass theorem. However, since the nonlinear term in problem \eqref{eq1} satisfies exponential growth, it is difficult to get the global Palais-Smale condition. To overcome the lack of compactness due to the presence of an exponential nonlinearity, we employ some tricks borrowed from  \cite{AFP}, where a critical Kirchhoff problem involving the fractional Laplacian has been studied. We first show that the energy functional associated with problem \eqref{eq1} satisfies the Palais-Smale condition at suitable levels $c_{\lambda}.$ In this process, the key point is to study the asymptotical behaviour of $c_{\lambda}$ as $\lambda\rightarrow\infty$, see Lemma \ref{lem3.4} for more details. For the case $1<p<N/s$ and $\lambda$ small enough, we prove that \eqref{eq1} has at least one nontrivial solution with negative energy by using Ekeland's variational principle. In order to get the multiplicity of solutions for problem \eqref{eq1}  for $\lambda$ small enough,
we follow some ideas from \cite{AA} and use the genus theory.

To the best of our knowledge, Theorems~\ref{th1}--\ref{th3} are the first results
for the Schr\"{o}dinger--Kirchhoff equations involving
Trudinger--Moser nonlinearities in the fractional setting.

The paper is organized as follows.
In Section~\ref{sec2}, we present the functional setting and prove
preliminary results.
In Section~\ref{sec3}, we obtain the existence of nontrivial nonnegative solutions for problem \eqref{eq1} for $\lambda$ large enough,
 by using the mountain pass theorem.
In Section~\ref{sec4},
we prove the existence of nonnegative solutions for problem \eqref{eq1} for $\lambda$ small enough,
 by using the Ekeland variational principle.
In Section~\ref{sec5},
we investigate the existence of infinitely many solutions for problem \eqref{eq1} by applying the genus theory.
In Section~\ref{sec6},
we extend Theorems \ref{th1}--\ref{th3} to get
wider applications, by replacing the fractional $N/s$--Laplacian
operator with a general
nonlocal integro--differential operator.
\medskip

\section{Preliminaries}\label{sec2}
In this section, we first provide the functional setting for problem \eqref{eq1}.
Let $1<p<\infty$ and let $L^p(\mathbb{R}^N,V)$ denote the Lebesgue space of real-valued functions, with
$V(x)|u|^p\in L^1({\mathbb{R}}^N),$ equipped with the norm
$$
\|u\|_{p,V}=\left(\int_{\mathbb{R}^N}V(x)|u|^p dx\right)^{{1}/{p}}\quad \text{for all }u\in L^p(\mathbb{R}^N,V).
$$
Set
\begin{align*}
W^{s,p}(\mathbb{R}^N)=\{u\in L^{p}(\mathbb{R}^N):
[u]_{s,p}<\infty\},
\end{align*}
where the Gagliardo seminorm $[u]_{s,p}$ is defined by
\begin{align*}
[u]_{s,p}=
\left(\iint_{\mathbb{R}^{2N}}\frac{|u(x)-u(y)|^{p}}{|x-y|^{N+sp}}dxdy\right)^{1/p}
.
\end{align*}
Equipped with the following norm
\begin{align*}
\|u\|_{s,p}=\left(\|u\|_{L^p(\mathbb{R}^N)}^p+[u]_{s,p}^p\right)^{1/p},
\end{align*}
$W^{s,p}(\mathbb{R}^N)$ is a Banach space.
The fractional critical exponent is defined by
\begin{align*}
p_s^*=
\begin{cases}
\frac{Np}{N-sp}\ \ &{\rm if}\ sp<N;\\
\infty\ \ &{\rm if}\ sp\geq N.
\end{cases}
\end{align*}
Moreover, the fractional Sobolev embedding states that
$W^{s,p}(\mathbb{R}^N)\hookrightarrow L^{p_s^*}(\mathbb{R}^N)$ is continuous
if $sp<N$, and $W^{s,p}(\mathbb{R}^N)$$\hookrightarrow L^{q}(\mathbb{R}^N)$ is continuous for all $p\leq q<\infty$ if $sp=N$. For a more detailed account of the properties of
$W^{s,p}(\mathbb{R}^N)$, we refer to \cite{r28}.

By $(V_1)$ and \cite[Theorem 6.9]{r28}, the embedding $W^{s,N/s}_V(\mathbb{R}^N)\hookrightarrow L^{\nu}(\mathbb{R}^N)$ is continuous for any $\nu\in [N/s,\infty)$, namely there exists a positive constant $C_\nu$ such that
\begin{align*}
\|u\|_{L^{\nu}(\mathbb{R}^N)}\leq C_\nu \|u\|\quad\mbox{for all } u\in W^{s,N/s}_V(\mathbb{R}^N).
\end{align*}
To prove the existence of weak solutions of~\eqref{eq1}, we shall
use the following embedding theorem.

\begin{theorem}[Compact embedding, II -- Theorem~2.1 of~\cite{PXZ}]
\label{th2.1}
Assume that conditions $(V_1)$ and $(V_2)$ hold. Then for any $\nu\geq N/s$
the embedding
$
W_V^{s,N/s}(\mathbb{R}^N)\hookrightarrow\hookrightarrow  L^\nu(\mathbb{R}^N)
$
is compact.
\end{theorem}
\begin{proof}
The proof is similar to the proof of Theorem 2.1 in \cite{PXZ}. Indeed, one can choose $\varrho>N/s$ such that $\nu\in [N/s,\varrho]$. Here $\varrho$ plays the same role as $p_s^*$ in \cite[Theorem 2.1]{PXZ}. Then, using the fact that the embedding $W^{s,N/s}_V(\mathbb{R}^N)\hookrightarrow L^{\varrho}(\mathbb{R}^N)$ is continuous and the same discussion as Theorem 2.1 in \cite{PXZ}, one can obtain the desired conclusion.
\end{proof}

The following radial lemma can be found in \cite[Radial Lemma A.IV]{BLions}.
\begin{lemma}\label{RL}
Let $N\geq2$. If $u\in L^p(\mathbb{R}^N)$ with $1\leq p<\infty$, is a radial non-increasing function (i.e. $ 0\leq u(x)\leq u(y)$ if $|x|\geq |y|$), then
\begin{align*}
|u(x)|\leq |x|^{-N/p}\left(\frac{N}{\omega_{N-1}}\right)^{1/p}\|u\|_{L^p(\mathbb{R}^N)},\  x\neq 0,
\end{align*}
where  $\omega_{N-1}$ is the $(N-1)$-dimensional measure of the $(N- 1)$-sphere.
\end{lemma}
Clearly, by Lemma \ref{RL}, we have
\begin{align}\label{RI}
|u(x)|^{N/s}\leq \frac{N\|u\|^{N/s}_{N/s, V}}{V_0\omega_{N-1}|x|^{N}},\  x\neq 0,
\end{align}
for all radial non-increasing function $u\in W_V^{s,N/s}(\mathbb{R}^N)$.

In the sequel, we will prove some technical lemmas which will be used later on.

\begin{lemma}\label{lem2.2}
Let $\alpha<\alpha_{N,s}$. If $u\in W_V^{s,N/s}(\mathbb{R}^N)$ and $\|u\|\leq K $ with $0<K<\left(\frac{\alpha_{N,s}}{\alpha}\right)^{(N-s)/N}$ and $\varphi\in W_V^{s,N/s}(\mathbb{R}^N)$. Then there exists a constant $C(N,s,\alpha,K)>0$ such that
\begin{align*}
\int_{\mathbb{R}^N}\Phi_\alpha(u)|\varphi|dx\leq C(N,s,\alpha,K)\|\varphi\|_{L^\nu(\mathbb{R}^N)},
\end{align*}
for some $\nu>N/s$.
\end{lemma}
\begin{proof}
Our proof is motivated by \cite{do}. We may assume $u,\varphi\geq0$, since we can replace $u,\varphi$ by $|u|$ and $|\varphi|$, respectively. To use the Schwarz symmetrization method, we briefly recall some  basic properties (see \cite{GPS}). Let $1\leq p<\infty$ and $u\in
L^p(\mathbb{R}^N)$ such that $u\geq 0$.
Then there is a unique nonnegative function $u^*\in
L^p(\mathbb{R}^N)$, called the Schwarz symmetrization of $u$, which
depends only on $|x|$, $u^*$ is a
decreasing function of $|x|$; and for all $\lambda>0$
\begin{align*}
|\{x:u^*(x)\geq \lambda\}|=|\{x:u(x)\geq \lambda\}|,
\end{align*}
and there exists $R_\lambda>0$ such that $\{x:u^*(x)\geq\lambda\}$ is the ball $B_{R_\lambda}(0)$
of radius $R_\lambda$ centered at origin. Moreover, for
any continuous and increasing function $G:[0,\infty)\rightarrow[0,\infty)$ such that $G(0)=0$,
\begin{align*}
\int_{\mathbb{R}^N}G(u^*(x))dx=
\int_{\mathbb{R}^N}G(u(x))dx.
\end{align*}
Furthermore, if $u\in W^{s,N/s}(\mathbb{R}^N)$, then
$u^*\in W^{s,N/s}(\mathbb{R}^N)$ and
\begin{align}\label{eq2.1}
\iint_{\mathbb{R}^{2N}}\frac{|u^*(x)-u^*(y)|^{N/s}}
{|x-y|^{2N}}dxdy\leq
\iint_{\mathbb{R}^{2N}}\frac{|u(x)-u(y)|^{N/s}}
{|x-y|^{2N}}dxdy,
\end{align}
see \cite{AL,Baernstein}.

According to the property of the Schwarz symmetrization (see \cite{AL, FB, CZ}), for $u,\varphi\in W_V^{s,N/s}(\mathbb{R}^N)$, we can conclude that
\begin{align*}
\int_{\mathbb{R}^N}\Phi_\alpha(u)|\varphi|dx
\leq\int_{\mathbb{R}^N}\Phi_\alpha(u^*)|\varphi^*|dx,
\end{align*}
where $u^*, \varphi^*$ are the Schwarz symmetrization of $u$ and $\varphi$, respectively. Applying H\"{o}lder's inequality, we get
\begin{align}\label{eq2.2}
\int_{\{|x|\leq R\}}\Phi_\alpha(u)|\varphi|dx
&\leq \int_{\mathbb{R}^N}\Phi_\alpha(u^*)|\varphi^*|dx\nonumber\\
&\leq \int_{\{|x|\leq R\}}\exp(\alpha|u^*|^{N/N-s})|\varphi^*|dx\nonumber\\
&\leq\left(\int_{\{|x|\leq R\}}\exp(r\alpha|u^*|^{N/N-s})dx\right)^{1/r}\left(\int_{|x|\leq R}|\varphi^*|^{\nu}dx\right)^{1/\nu},
\end{align}
where $r>1$ is sufficiently close to 1 so that $r\alpha<\alpha_{N,s}$ and $\nu=r(r-1)^{-1}>N/s$ and $R>0$ is a number to be determined later.

Let us recall two elementary inequalities. Since the function $g:[0,\infty)\rightarrow\mathbb{R}$ given by
\begin{align*}
g(t)=
\begin{cases}
\left[(t+1)^{\frac{N}{N-s}}-t^{\frac{N}{N-s}}-1\right]\big/t^{\frac{s}{N-s}},\ &\mbox{if}\ t\neq0,\\
0, \ \ &\mbox{if}\ t=0.
\end{cases}
\end{align*}
 is bounded on $[0,\infty)$, there exists $A=A(N,s)>0$ such that
\begin{align}\label{eq2.3}
(u+v)^{\frac{N}{N-s}}\leq
u^{\frac{N}{N-s}}+Au^{\frac{s}{N-s}}v+v^{\frac{N}{N-s}},\ \forall u,v\geq 0.
\end{align}
If $\epsilon$ and $\epsilon^\prime$ are positive real numbers
such that $\epsilon+\epsilon^\prime=1$, then for all $\varepsilon>0$, by the Young inequality we have
\begin{align}\label{eq2.4}
u^\epsilon v^{\epsilon^\prime}\leq\varepsilon u+\varepsilon^{-\frac{\epsilon}{\epsilon^\prime}} v,\ \ \forall u,v\geq 0.
\end{align}

Let
\begin{align*}
v(x)=\begin{cases}
u^*(x)-u^*(Rx_0)\ \ &{\rm if}\ x\in B_R(0),\\
0\ \  &{\rm if}\ x\in\mathbb{R}^N\setminus B_R(0),
\end{cases}
\end{align*}
where $x_0$ is some fixed point in $\mathbb{R}^N$ with $|x_0|=1$.
If $x\in\mathbb{R}^N\setminus B_R$ and $y\in B_R$, then
\begin{align*}
u^*(x)\leq u^*(y)\ \ {\rm and}\  u^*(x)\leq u^*(Rx_0)\leq u^*(y),
\end{align*}
since $u^*(x)$ is a decreasing function with respect to $|x|$.
Thus,
\begin{align*}
&\iint_{\mathbb{R}^{2N}}\frac{|v(x)-v(y)|^{N/s}}{|x-y|^{2N}}dxdy\\
&=\int_{B_R}\int_{B_R}\frac{|u^*(x)-u^*(y)|^{N/s}}{|x-y|^{2N}}dxdy
+2\int_{B_R}\int_{\mathbb{R}^N\setminus B_R}\frac{|u^*(x)-u^*(Rx_0)|^{N/s}}{|x-y|^{2N}}dxdy\\
&\leq \int_{B_R}\int_{B_R}\frac{|u^*(x)-u^*(y)|^{N/s}}{|x-y|^{2N}}dxdy
+2\int_{B_R}\int_{\mathbb{R}^N\setminus B_R}\frac{|u^*(x)-u^*(y)|^{N/s}}{|x-y|^{2N}}dxdy\\
&=\iint_{\mathbb{R}^{2N}}\frac{|u^*(x)-u^*(y)|^{N/s}}{|x-y|^{2N}}dxdy,
\end{align*}
which means that $v\in W_0^{s,N/s}(\Omega)$.

 By \eqref{eq2.3} and \eqref{eq2.4}, we obtain
\begin{align*}
|u^*(x)|^{\frac{N}{N-s}}=|v+u^*(Rx_0)|^{\frac{N}{N-s}}
\leq |v|^{\frac{N}{N-s}}+A|v|^{\frac{s}{N-s}}u^*(Rx_0)
+|u^*(Rx_0)|^{\frac{N}{N-s}},
\end{align*}
and
\begin{align*}
v^{\frac{s}{N-s}}u^*(Rx_0)
=(v^{\frac{N}{N-s}})^{\frac{s}{N}}[(u^*(Rx_0))^{\frac{N}{N-s}}]^{\frac{N-s}{N}}
\leq \frac{\varepsilon}{A}|v|^{\frac{N}{N-s}}+\left(\frac{\varepsilon}{A}\right)^{-\frac{s}{N-s}}
(u^*(Rx_0))^{\frac{N}{N-s}}.
\end{align*}
It follows that
\begin{align*}
|u^*|^{\frac{N}{N-s}}\leq (1+\varepsilon)|v|^{\frac{N}{N-s}}+
C(\varepsilon,s,N)(u^*(Rx_0))^{\frac{N}{N-s}},
\end{align*}
where $C(\varepsilon,s,N)=A^{\frac{N}{N-s}}\varepsilon^{\frac{s}{s-N}}
+1$. Therefore,
\begin{align*}
\int_{|x|<R}\exp(\alpha r|u^*|^{\frac{N}{N-s}})dx
&\leq\exp\left(C(\varepsilon,s,N)u^*(Rx_0)^{\frac{N}{N-s}}\right)
\int_{|x|<R}\exp(r\alpha|(1+\varepsilon)v|^{\frac{N}{N-s}})dx\\
&\leq\exp\left(C(\varepsilon,s,N)u^*(Rx_0)^{\frac{N}{N-s}}\right)
\int_{|x|<R}\exp\left(r\alpha((1+\varepsilon)\|v\|)^{\frac{N}{N-s}}|\frac{v}{\|v\|}|^{\frac{N}{N-s}}\right)dx.
\end{align*}
Choosing $\varepsilon>0$ and $K$ small enough such that $$r[(1+\varepsilon)\|v\|]^{\frac{N}{N-s}}\alpha
\leq r[(1+\varepsilon)\|u^*\|]^{\frac{N}{N-s}}\alpha\leq r[(1+\varepsilon)K]^{\frac{N}{N-s}}\alpha<
\alpha_{N,s},$$
we get
\begin{align*}
\int_{|x|<R}\exp(r\alpha|(1+\varepsilon)v|^{\frac{N}{N-s}})dx\leq C_{N,s}|B_R(0)|,
\end{align*}
thanks to the Trudinger-Moser inequality on bounded domains.
Hence, we obtain
\begin{align*}
\int_{|x|<R}\exp(\alpha r|u^*|^{\frac{N}{N-s}})dx
\leq C_{N,s}\frac{\omega_{N-1}R^N}{N}\exp\left(C(\varepsilon,s,N)u^*(Rx_0)^{\frac{N}{N-s}}\right).
\end{align*}
Furthermore, \eqref{RI} yields that
\begin{align*}
\int_{|x|<R}\exp(\alpha r|u^*|^{\frac{N}{N-s}})dx
\leq C_{N,s}\frac{\omega_{N-1}R^N}{N}
\exp\left(C(\varepsilon,s,N)|R|^{-\frac{sN}{N-s}}
\left(\frac{N}{\omega_{N-1}}\right)^{\frac{s}{N-s}}K^{\frac{N}{N-s}}\right).
\end{align*}
Therefore, by \eqref{eq2.2}, we arrive at
\begin{align}\label{eq2.5}
\int_{|x|< R}\Phi_\alpha(u)|\varphi|dx\leq C(N,\alpha,s,K)\|\varphi\|_{L^\nu(\mathbb{R}^N)}.
\end{align}

On the other hand, we have
\begin{align*}
\int_{|x|\geq R}\Phi_\alpha(u^*)|\varphi^*|dx
=\int_{|x|\geq R}\sum_{j=k_0}^\infty\frac{\alpha^j|u^*|^{\frac{N}{N-s}j}|\varphi^*|}{j!}dx,
\end{align*}
where $k_0$ is the smallest integer such that $k_0\geq p-1$.
Using \eqref{RI}, $\|u^*\|\leq\|u\|\leq K$ and the H\"{o}lder inequality, we get
\begin{align}\label{eq2.7}
&\int_{|x|\geq R}|u^*|^{Nj/(N-s)}|\varphi^*|dx \nonumber\\
&\leq \left(\left(\frac{N}{V_0\omega_{N-1}}\right)^{s/N}\|u^*\|\right)^{Nj/(N-s)}
\int_{|x|\geq R}\frac{|\varphi^*|}{|x|^{Nsj/(N-s)}}dx\nonumber\\
&\leq \left(\left(\frac{N}{V_0\omega_{N-1}}\right)^{s/N}\|u^*\|\right)^{Nj/(N-s)}
\left(\int_{|x|\geq R}\frac{1}{|x|^{rNsj/(N-s)}}dx\right)^{1/r}
\left(\int_{|x|\geq R}|\varphi^*|^\nu dx\right)^{1/\nu}\nonumber\\
&\leq \omega_{N-1}R^N\left(\left(\frac{N}{V_0\omega_{N-1}}\right)^{s/N}R^{-rs}K\right)^{Nj/(N-s)}\|\varphi\|_{L^\nu(\mathbb{R}^N)},
\end{align}
for all $j\geq k_0>p-1$.
Choosing
\begin{align*}
R^{rs}=K\left(\frac{N}{V_0\omega_{N-1}}\right)^{s/N},
\end{align*}
we can deduce from \eqref{eq2.7} that
\begin{align}\label{eq2.8}
\int_{|x|\geq R}|u^*|^{Nj/(N-s)}|\varphi^*|dx\leq C(N,s,K)\|\varphi\|_{L^\nu(\mathbb{R}^N)}\ \ {\rm for\ all}\ j\geq k_0.
\end{align}
For the case $k_0=p-1$, we have by the H\"{o}lder inequality,
\begin{align}\label{eq2.9}
\int_{|x|\geq R}|u^*|^{N/s}|\varphi^*|dx
&\leq \left(\int_{|x|\geq R}|u^*|^{rN/s}dx\right)^{1/r}
 \left(\int_{|x|\geq R}|\varphi^*|^{\nu}dx\right)^{1/\nu}\nonumber\\
&\leq C(N,s,K)\|\varphi\|_{L^\nu(\mathbb{R}^N)}.
\end{align}
Here we have used the continuous embedding $W_{V}^{s,N/s}(\mathbb{R}^N)\hookrightarrow L^{Nr/s}(\mathbb{R}^N)$.

Combining \eqref{eq2.8} and \eqref{eq2.9}, we can conclude that
\begin{align*}
\int_{|x|\geq R}\Phi_\alpha(u^*)|\varphi^*|dx&=\sum_{j=k_0}^\infty\frac{\alpha^j}{j_!}\int_{|x|\geq R}|u^*|^{Nj/(N-s)}|\varphi^*|dx\\
&\leq C(N,s,K)\|\varphi\|_{L^\nu(\mathbb{R}^N)}\sum_{j=k_0}^\infty\frac{\alpha^j}{j_!}\\
&\leq C(N,s,K)\|\varphi\|_{L^\nu(\mathbb{R}^N)}\exp(\alpha_{N,s}),
\end{align*}
which together with \eqref{eq2.5} yields the desired result.
\end{proof}
Similarly, we can obtain the following lemma.
\begin{lemma}\label{lem2.3}
Let $\alpha<\alpha_{N,s}$, $N/s<q$, $u\in W_V^{s,N/s}(\mathbb{R}^N)$ and $\|u\|\leq K$ with $0<K<\left(\frac{\alpha_{N,s}}{\alpha}\right)^{(N-s)/N}$. Then there exists $C(N,s,\alpha,K)>0$ such that
\begin{align*}
\int_{\mathbb{R}^N}\Phi_\alpha(u)|u|^qdx\leq C(N,s,\alpha,K)\|u\|^q.
\end{align*}
\end{lemma}
To study the nonnegative solutions of equation \eqref{eq1}, we define the associated functional $I_\lambda:W_{V}^{s,N/s}(\mathbb{R}^N)\to\mathbb{R}$ as
\begin{align*}
& I_\lambda(u)=\frac{s}{N}\mathscr{M}(\|u\|^{N/s})
-\int_{\mathbb{R}^N}F(x,u)dx-
\lambda\int_{\mathbb{R}^N}h(x)|u^+|^{p}dx,
\end{align*}
where $u^+=\max\{u,0\}$.
Under the assumption $(f_1)$ and the fractional Trudinger-Moser inequality, one can verify that $I_\lambda$ is well defined, of class $C^1(W_{V}^{s,N/s}(\mathbb{R}^N),\mathbb{R})$, and
\begin{align*}
\langle I_\lambda^\prime(u),v\rangle=
M(\|u\|^{N/s})&\left(\langle u,v\rangle_{s,N/s}+\int_{\mathbb{R}^N}V(x)|u|^{N/s-2}uvdx\right)\\
&-\int_{\mathbb{R}^N}f(x,u)vdx-
\lambda\int_{\mathbb{R}^N}h|u^+|^{p-2}u^+vdx,
\end{align*}
for all $u,v\in W_V^{s,N/s}(\mathbb{R}^N)$. Hereafter, $\langle\cdot,\cdot\rangle$ denotes the
duality pairing between $\big(W_V^{s,N/s}(\mathbb{R}^N)\big)'$ and
$W_V^{s,N/s}(\mathbb{R}^N)$.
Clearly, the critical points of $I_\lambda$ are exactly the weak solutions of equation \eqref{eq1}.
Moreover, the following lemma shows that every nontrivial weak solution of problem \eqref{eq1} is nonnegative.
\begin{lemma}\label{lem3.1}
Let $(M_1)$ and $(f_1)$ hold. If $h(x)\geq 0$ for almost every $x\in\mathbb{R}^N$, then for all $\lambda>0$ any nontrivial critical point of
functional $ I_\lambda$ is nonnegative.
\end{lemma}

\begin{proof}
Fix $\lambda>0$ and let $u_\lambda\in W_V^{s,N/s}(\mathbb{R}^N)\setminus\{0\}$ be a critical point of functional $ I_\lambda$. Clearly,
$u_\lambda^-=\max\{-u,0\}\in W_V^{s,N/s}(\mathbb{R}^N)$. Then
$\langle  I_\lambda^\prime(u_\lambda),-u_\lambda^-\rangle=0$, a.e.
\begin{align*}
M(\|u_\lambda\|^{N/s})&\left(\langle u_\lambda,u_\lambda^-\rangle_{s,N/s}
+\int_{\mathbb{R}^N}V|u_\lambda|^{N/s-2}u_\lambda(-u_\lambda^-)dx\right)\\
&\qquad=\int_{\mathbb{R}^N}f(x,u_\lambda)(-u_\lambda^-)dx
+\lambda\int_{\mathbb{R}^N}h|u_\lambda^+|^{p-2}u_\lambda^+(-u_\lambda^-)dx.
\end{align*}
We observe that for a.e. $x, y\in\mathbb{R}^N$,
\begin{align*}
&|u_\lambda(x)-u_\lambda(y)|^{N/s-2}(u_\lambda(x)-u_\lambda(y))(-u_\lambda^-(x)+u_\lambda(y)^-)\\
=&|u_\lambda(x)-u_\lambda(y)|^{N/s-2}u_\lambda^+(x)u_\lambda^-(y)
+|u_\lambda(x)-u_\lambda(y)|^{N/s-2}u_\lambda^-(x)u_\lambda^+(y)+[u_\lambda^-(x)-u_\lambda^-(y)]^{N/s}\\
\geq &|u_\lambda^--u_\lambda^-(y)|^{N/s},
\end{align*}
$f(x,u_\lambda)u_\lambda^-=0$ a.e. $x\in\mathbb{R}^N$ by $(f_1)$ and
\begin{align*}
\int_{\mathbb{R}^N}V|u_\lambda(x)|^{N/s-2}u_\lambda(-u_\lambda^{-}(x))dx=\int_{\mathbb{R}^N}V|u_\lambda^-|^{N/s}dx.
\end{align*}
Moreover,  $h|u_\lambda^+|^{p-2}u_\lambda^+(-u_\lambda^-)\leq0$ a.e. in
$\mathbb{R}^N$. Hence,
\begin{align*}
M(\|u_\lambda\|^{N/s})\left([u_\lambda^-]_{s,N/s}^{N/s}+\|u_\lambda^-\|_{N/s,V}^{N/s}\right)\leq 0.
\end{align*}
This, together with $\|u_\lambda\|>0$ and $(M_1)$, implies that $u_\lambda^-\equiv0$, that is $u_\lambda\geq 0$ a.e. in $\mathbb{R}^N$. This completes the proof.
\end{proof}

\section{Proof of Theorem \ref{th1}}\label{sec3}

Let us recall that  $ I_\lambda$ satisfies the $(PS)_c$ {\em condition} in $W_V^{s,N/s}(\mathbb{R}^N)$, if any $(PS)_c$
	sequence $\{u_n\}_{n}\subset W_V^{s,N/s}(\mathbb{R}^N)$, namely
a sequence such that $I_\lambda(u_n)\rightarrow c$ and
$ I_\lambda^\prime (u_n)\rightarrow 0$ as $n\rightarrow\infty$, admits a strongly convergent
subsequence in $W_V^{s,N/s}(\mathbb{R}^N)$.

In the sequel, we shall make use of the following general mountain pass theorem (see \cite{AE}).
\begin{theorem}\label{MPth}
Let $E$ be a real Banach space and $J\in C^1(E,\mathbb{R})$ with $J(0)=0$. Suppose that
\begin{itemize}
\item[$(i)$] there exist $\rho,\alpha>0$ such that $J(u)\geq \alpha$ for all $u\in E$, with $\|u\|_{E}=\rho$;
\item[$(ii)$] there exists $e\in E$ satisfying $\|e\|_{E}>\rho$ such that $J(e)<0$.
\end{itemize}
Define $\Gamma=\{\gamma\in C^1([0,1];E):\gamma(0)=1,\gamma(1)=e\}$. Then
$$c=\inf_{\gamma\in\Gamma}\max_{0\leq t\leq 1} J(\gamma(t))\geq\alpha$$
and there exists a $(PS)_c$ sequence $\{u_n\}_n\subset E$.
\end{theorem}

To find a mountain pass solution of problem~\eqref{eq1}, let us first verify the validity of the conditions of Theorem~\ref{MPth}.

\begin{lemma}[Mountain Pass Geometry 1]\label{lem3.2} Assume that $(V_1)$, $(f_1)$, and $(f_3)$ hold. Then for each $\lambda>0$, there exist $\rho_\lambda>0$ and $\kappa>0$ such that
$I_\lambda(u)\geq\kappa>0$ for any $u\in W_V^{s,N/s}(\mathbb{R}^N)$, with $\|u\|=\rho_\lambda$.
\end{lemma}

\begin{proof}
By $(f_3)$,  there exists $\tau,\delta,C_\varepsilon>0$ such that for all $|u|\leq\delta$
\begin{align}\label{c4}
|F(x,u)|\leq \frac{s\mathscr{M}(1)}{N}(\lambda_1-\tau)|u|^{N/s}\quad\mbox{for all } x\in\mathbb{R}^N.
\end{align}
Moreover, by $(f_1)$, for each $q>N/s$, we can find a constant $C=C(q,\delta)>0$ such that
\begin{align}\label{c5}
F(x,u)\leq C|u|^q\Phi_\alpha(u)
\end{align}
for all $|u|\geq \delta$ and $x\in\mathbb{R}^N$. Combining \eqref{c4}
and \eqref{c5}, we obtain
\begin{align}\label{c6}
F(x,u)\leq \frac{s\mathscr{M}(1)}{N}(\lambda_1-\tau)|u|^{N/s}+C|u|^q\Phi_\alpha(u)
\end{align}
for all $u\in\mathbb{R}$ and $x\in\mathbb{R}^N$.

On the other hand, $(M_2)$ gives
 \begin{align}\label{M1}
 \mathscr{M}(t)\geq \mathscr{M}(1)t^\theta\quad\mbox{for all }t\in [0,1].
 \end{align}
Thus, by using  \eqref{c6}, \eqref{M1} and the H\"{o}lder inequality,  we obtain for all $u\in W_V^{s,N/s}(\mathbb{R}^N)$, with $\|u\|\leq 1$ small enough,
\begin{align*}
I_\lambda(u)&\geq\frac{s\mathscr{M}(1)}{N}\|u\|^{\theta N/s}-\frac{s\mathscr{M}(1)}{N}\frac{(\lambda_1-\tau)}{\lambda_1}\|u\|^{\theta N/s}- C(N,s,\alpha)\|u\|^q-\lambda
|h|_\infty\|u\|_{L^p(\mathbb{R}^N)}^p\nonumber\\
&\geq \frac{\tau s\mathscr{M}(1)}{\lambda_1 N}\|u\|^{\theta N/s}-C(N,s,\alpha)\|u\|^q-\lambda
|h|_\infty S_p^p\|u\|^p,
\end{align*}
where $S_p$ is the best constant from embedding $W_{V}^{s,N/s}(\mathbb{R}^N)$ to $L^p(\mathbb{R}^N)$.
Since $1<\theta N/s<p,q$, we can choose $\rho\in (0,1)$ such that $\frac{\tau s\mathscr{M}(1)}{\lambda_1 N}\rho^{\theta N/s}-C(N,s,\alpha)\rho^{q}-\lambda
|h|_\infty S_p^p\rho^p>0$. Thus, $I_\lambda(u)\geq \kappa:=
\frac{\tau s\mathscr{M}(1)}{\lambda_1 N}\rho^{\theta N/s}-C(N,s,\alpha)\rho^{q}-\lambda
|h|_\infty S_p^p\rho^p>0$ for all $u\in W_V^{s,N/s}(\mathbb{R}^N)$, with $\|u\|=\rho$.
\end{proof}

\begin{lemma}[Mountain Pass Geometry 2]\label{lem3.3} Assume that $(f_1)$--$(f_2)$ hold.
Then there exists a nonnegative function $e\in C_0^\infty(\mathbb{R}^N)$, independent of $\lambda$, such that $I_\lambda(e)<0$ and $\|e\|\geq\rho_\lambda$ for all $\lambda\in\mathbb{R}^+$.
\end{lemma}

\begin{proof}
It follows from  $(M_2)$ that
\begin{align}\label{M2}
\mathscr{M}(t)\leq \mathscr{M}(1)t^\theta\quad\mbox{for  all } t\geq 1.
\end{align}
Furthermore, $F(x,t)\geq0$ for all $(x,t)\in\mathbb{R}^N\times \mathbb{R}$
by $(f_1)$ and $(f_2)$.
Let $u\in W^{s,N/s}_V(\mathbb{R}^N)\setminus\{0\},$ $u\geq 0$ with compact support $\Omega={\rm supp} (u)$ and $\|u\|=1$. By $(f_2)$, we obtain that for $\mu>\theta N/s$, there exist positive constants $C_1,C_2>0$ such that
\begin{align}\label{eq3.6}
F(x,t)\geq C_1 t^{\mu}-C_2\ \ {\rm for\ all}\ x\in\Omega\ {\rm and}\ t\geq0.
\end{align}
Then for all $t\geq 1$, we have
\begin{align*}
 I_\lambda(tu)&\leq \frac{s}{N}\mathscr{M}(1)t^{\theta N/s}\|u\|^{\theta N/s}
-C_1t^{\mu}\int_{\Omega}|u|^\mu dx+C_2|\Omega|.
\end{align*}
Hence, $ I_\lambda(tu)\rightarrow-\infty$ as $t\rightarrow\infty$, since $\theta N/s<\mu$. The lemma is now proved by taking $e=Tu$, with $T>0$ so large that
$\|e\|\geq \rho_\lambda$ and $I_\lambda(e)<0$.
\end{proof}

By Theorem \ref{MPth}, there exists a $(PS)_c$ sequence $\{u_n\}_n\subset W_V^{s,N/s}(\mathbb{R}^N))$ such that
\begin{align*}
I_\lambda(u_n)\rightarrow c_\lambda \ \ {\rm and}\ I_\lambda^\prime(u_n)\rightarrow 0\ \ {\rm as}\ n\rightarrow\infty,
\end{align*}
where
\begin{align}\label{eq3.77}
c_{\lambda}=\inf_{\gamma\in\Gamma}\max_{t\in[0,1]}I_\lambda(\gamma(t)),
\end{align}
and $\Gamma=\left\{\gamma\in C([0,1];W_V^{s,N/s}(\mathbb{R}^N)):\gamma(0)=0,\
 I_\lambda(\gamma(1))=e\right\}.$
Obviously, $c_\lambda>0$ by Lemma \ref{lem3.1}. Moreover, we have the following result.

\begin{lemma}\label{lem3.4}
Suppose that $V$ satisfies $(V_1)$--$(V_2)$ and $f$ satisfies $(f_1)$--$(f_2)$. Then
$$\lim_{\lambda\rightarrow\infty}c_{\lambda}=0,$$
where $c_{\lambda}$ is given by \eqref{eq3.77}.
\end{lemma}

\begin{proof}
For $e$ given by Lemma~\ref{lem3.3}, we have $\lim_{t\rightarrow\infty} I_\lambda(te)=-\infty$. Therefore, there exists $t_{\lambda}>0$ such that $ I_\lambda(t_{\lambda}e)=\max_{t\geq0}I_\lambda(te)$. Hence, by $ I_\lambda^\prime(t_\lambda e)=0$, we have
\begin{align}\label{eq3.7}
t_{\lambda}^{N/s}M(\|t_\lambda e\|^{N/s})\|e\|^{N/s}=\int_{\mathbb{R}^N}f(x,t_{\lambda}e)t_\lambda edx+\lambda t_{\lambda}^{p}\int_{\mathbb{R}^N}
 he^{p}dx.
\end{align}

Let us first claim that $\{t_\lambda\}_\lambda$ is bounded. Arguing by contradiction, we assume that there exists a subsequence of $\{t_\lambda\}_\lambda$ still denoted by $\{t_\lambda\}_\lambda$ such that $t_\lambda\rightarrow \infty$ as $\lambda\rightarrow\infty$. Then for $\lambda$ large enough, by $(M_2)$ and \eqref{M2} we get
\begin{align*}
\theta \mathscr{M}(1)t_\lambda^{N\theta/s}\|e\|^{N\theta/s}\geq
\lambda t_{\lambda}^{p}\int_{\mathbb{R}^N}
 he^{p}dx,
\end{align*}
thanks to $(f_2)$. It follows from $p>N\theta/s$ that $t_\lambda\rightarrow 0$ as $\lambda\rightarrow\infty$ which is a contradiction. Hence $\{t_\lambda\}_\lambda$ is bounded.

Therefore, up to a subsequence, one can prove that $t_\lambda\rightarrow 0$ as $\lambda\rightarrow\infty$.
Put $\overline{\gamma}(t)=te$. Clearly, $\overline{\gamma}\in\Gamma$, thus by the continuity of $\mathscr{M}$, we have
\begin{align*}
0<c_{\lambda}\leq\max_{t\geq0} I_\lambda(\overline{\gamma}(t))
= I_\lambda(t_{\lambda}e)&\leq \frac{1}{p}\mathscr{M}(\|t_{\lambda}e\|^{N/s})
\rightarrow 0
\end{align*}
as $\lambda\rightarrow\infty$.
The lemma is now proved.
\end{proof}

\begin{lemma}\label{lem3.5}
Let $\{u_n\}_n\subset W_V^{s,N/s}(\mathbb{R}^N)$ be a $(PS)_{c_\lambda}$ sequence associated with $I_\lambda$. Then there exists $\Lambda_1>0$ such that for all $\lambda>\Lambda_1$, up to a subsequence still denoted by $\{u_n\}_n$,
\begin{align*}
\limsup_{n\rightarrow\infty}\|u_n\|<\left(\frac{\alpha_{N,s}}{\alpha_0}\right)^{\frac{N-s}{N}}.
\end{align*}
\end{lemma}
\begin{proof}
We first claim that $\{u_n\}_n$ is bounded in $W_V^{s,N/s}(\mathbb{R}^N)$. Indeed, this follows from the fact that $\{u_n\}_n$ is a $(PS)_c$ sequence such that
\begin{align*}
I_\lambda(u_n)-\frac{1}{\mu}\langle I_\lambda^\prime(u_n),u_n\rangle\leq c+o(1)+o(1)\|u_n\|.
\end{align*}
On the other hand, if $N\theta/s<\mu\leq p$, by $(f_2)$ we have
\begin{align*}
I_\lambda(u_n)-\frac{1}{\mu}\langle I_\lambda^\prime(u_n),u_n\rangle&=\left(\frac{s}{N\theta}-\frac{1}{\mu}\right)M(\|u_n\|^{N/s})\|u_n\|^{N/s}
\\&\ \ +\int_{\mathbb{R}^N}\left(\frac{1}{\mu}f(x,u_n)u_n-F(x,u_n)\right)dx
+\left(\frac{1}{\mu}-\frac{1}{p}\right)\lambda\int_{\mathbb{R}^N}h(x)|u_n|^pdx\\
&\geq \left(\frac{s}{N\theta}-\frac{1}{\mu}\right)M(\|u_n\|^{N/s})\|u_n\|^{N/s}.
\end{align*}
Then
\begin{align}\label{eq3.9}
\left(\frac{s}{N\theta}-\frac{1}{\mu}\right)M(\|u_n\|^{N/s})\|u_n\|^{N/s}\leq c_\lambda+o(1)+o(1)\|u_n\|
\end{align}
which means that $\{u_n\}_n$ is bounded in $W_V^{s,N/s}(\mathbb{R}^N)$.
Similarly, if $\mu>p$, we obtain
\begin{align*}
I_\lambda(u_n)-\frac{1}{p}\langle I_\lambda^\prime(u_n),u_n\rangle&=\left(\frac{s}{N\theta}-\frac{1}{\mu}\right)M(\|u_n\|^{N/s})\|u_n\|^{N/s}
\\&\ \ +\int_{\mathbb{R}^N}\left(\frac{1}{p}f(x,u_n)u_n-F(x,u_n)\right)dx\\
&\geq \left(\frac{s}{N\theta}-\frac{1}{p}\right)M(\|u_n\|^{N/s})\|u_n\|^{N/s}.
\end{align*}
Then
\begin{align}\label{eq3.10}
\left(\frac{s}{N\theta}-\frac{1}{p}\right)M(\|u_n\|^{N/s})\|u_n\|^{N/s}\leq c_\lambda+o(1)+o(1)\|u_n\|.
\end{align}

 If $d:=\inf_{n\geq1}\|u_n\|>0$, then by $(M_1)$, \eqref{eq3.9} and \eqref{eq3.10}  we get
\begin{align*}
\limsup_{n\rightarrow\infty}
\|u_n\|\leq \max\left\{\left(\frac{\mu N\theta\kappa(d)}{\mu s-N\theta}
\right)^{\frac{s}{
N}}, \left(\frac{p N\theta\kappa(d)}{ps-N\theta}
\right)^{\frac{s}{
N}}\right\}c_\lambda^{\frac{s}{
N}}.
\end{align*}
Hence by Lemma \ref{lem3.4}, there exists $\Lambda_1>0$ such that for all $\lambda>\Lambda_1$
\begin{align*}
\limsup_{n\rightarrow\infty}\|u_n\|<\left(\frac{\alpha_{N,s}}{\alpha_0}
\right)^{\frac{N-s}{N}}.
\end{align*}
If $d:=\inf_{n\geq1}\|u_n\|=0$, we can take a subsequence of $\{u_n\}_n$ such that the result holds. Thus, the proof is complete.
\end{proof}
\begin{lemma}[The $(PS)_{c_\lambda}$ condition]\label{lem3.6}
Let $(V_1)$--$(V_2)$ and $(f_1)$--$(f_2)$ hold. Then the functional $I_\lambda$ satisfies the $(PS)_{c_\lambda}$ condition for all $\lambda>\Lambda_1.$
\end{lemma}

\begin{proof}
Let $\{u_n\}_n$ be a $(PS)_{c_\lambda}$ sequence. Then by Lemma \ref{lem3.5}, passing to a subsequence, if necessary, we obtain $\|u_n\|\rightarrow \eta\geq0$. If $\eta=0$, then the proof is complete. Thus, in the sequel we can assume that $\eta>0$. Then for $n$ large, $\|u_n\|\geq \frac{1}{2}\eta>0$.

Next, we show that $\{u_n\}_n$ has a convergent subsequence in $W_{V}^{s,N/s}(\mathbb{R}^N)$. By Lemma \ref{lem3.5} and Theorem \ref{th2.1}, passing if necessary to a subsequence, we can assume that
\begin{align}
&u_n\rightharpoonup u\ \ {\rm weakly\ in}\ W_{V}^{s,N/s}(\mathbb{R}^N),\nonumber\\
& u_n\rightarrow u\ \ {\rm strongly\ in}\ L^{\nu}(\mathbb{R}^N)( \nu\geq N/s),\nonumber\\
&u_n\rightarrow u\ \ {\rm a.e.\ in}\ \mathbb{R}^N.
\end{align}
Since $\{u_n\}_n$ is a bounded $(PS)_c$ sequence in $W_{V}^{s,N/s}(\mathbb{R}^N)$, we have
\begin{align}\label{eq3.11}
o(1)&=\langle I_\lambda^\prime(u_n),u_n-u\rangle
\nonumber\\ &=M(\|u_n\|^{N/s})\left(\langle u_n,u_n-u\rangle_{s,N/s}
+\int_{\mathbb{R}^N}V|u_n|^{\frac{N}{s}-2}u_n(u_n-u)dx\right)
 \nonumber \\&-\int_{\mathbb{R}^N}\left[f(x,u_n)(u_n-u)+\lambda h(x)|u_n^+|^{p-2}u_n^+(u_n-u)\right]dx.
\end{align}
Define a functional $L$ as follows
\begin{align*}
\langle L(v),w\rangle=\langle v,w\rangle_{s,N/s}+\int_{\mathbb{R}^N}V(x)|v|^{\frac{N}{s}-2}vwdx
\end{align*}
for all $v,w\in W_V^{s,N/s}(\mathbb{R}^N)$. By the H\"{o}lder inequality, one can see that
\begin{align*}
|\langle L(v),w\rangle|\leq \|v\|^{N/s-1}\|w\|,
\end{align*}
which together with the definition of $L$ implies that
for each $v$, $L(v)$ is a bounded linear functional on $W_V^{s,N/s}(\mathbb{R}^N)$.
Thus, $\langle L(u),u_n-u\rangle=o(1)$, that is,
\begin{align*}
\langle u,u_n-u\rangle_{s,N/s}
+\int_{\mathbb{R}^N}V(x)|u|^{N/s-2}u(u_n-u)dx=o(1).
\end{align*}
Similarly, one can deduce that
\begin{align*}
\lim_{n\rightarrow\infty}\int_{\mathbb{R}^N}h(x)|u_n^+|^{p-2}u_n^+(u_n-u)dx=0.
\end{align*}
Using assumptions $(f_1)$ and $(f_2)$, we have
\begin{align*}
\left|\int_{\mathbb{R}^N}f(x,u_n)(u_n-u)dx\right|\leq b_1\int_{\mathbb{R}^N}|u_n|^{N\theta /s-1}|u_n-u|dx
+b_2\int_{\mathbb{R}^N}\Phi_\alpha(u_n)|u_n-u|dx.
\end{align*}
Further, by the Holder inequality, we get
\begin{align*}
\int_{\mathbb{R}^N}|u_n|^{N\theta /s-1}|u_n-u|dx
&\leq \|u_n\|_{L^{N\theta  /s}(\mathbb{R}^n)}^{N\theta /s-1}
\|u_n-u\|_{L^{N\theta /s}(\mathbb{R}^N)}\\
&\leq C\|u_n-u\|_{L^{N\theta /s}(\mathbb{R}^N)}\rightarrow 0\ \ {\rm as}\ n\rightarrow\infty.
\end{align*}
On the other hand, by Lemmas \ref{lem2.2} and \ref{lem3.5}, for some $\nu>N\theta /s$ we obtain
\begin{align*}
&\int_{\mathbb{R}^N}\Phi_\alpha(u_n)|u_n-u|dx\\
&\leq C\|u_n-u\|_{L^{\nu}(\mathbb{R}^N)}\rightarrow0\ \ {\rm as}\ n\rightarrow\infty.
\end{align*}

In conclusion, we can deduce from \eqref{eq3.11} that
\begin{align*}
&M(\|u_n\|^{N/s})\bigg[\langle u_n,u_n-u\rangle_{s,N/s}-\langle u,u_n-u\rangle_{s,N/s}\\
&\ \ +\int_{\mathbb{R}^N}V(x)\left(|u_n|^{N/s-2}u_n-|u|^{N/s-2}u\right)(u_n-u)dx
\bigg]=o(1).
\end{align*}
By using the following inequality:
\begin{align*}
(|\xi|^{N/s-2}\xi-|\eta|^{N/s-2}\eta)\cdot(\xi-\eta)
\geq C_{s,N}|\xi-\eta|^{N/s},\ \ N\geq 2>2s,
\end{align*}
we can easily obtain that $\|u_n-u\|\rightarrow 0$ as $n\rightarrow\infty$. Thus, the proof is complete.
\end{proof}

\begin{proof}[\bf Proof of Theorem \ref{th1}]
By Lemmas~\ref{lem3.2} and \ref{lem3.3},  we know that $ I_\lambda$ satisfies all  assumptions of Theorem~\ref{MPth}.
Hence there exists a $(PS)_{c_\lambda}$ sequence. Moreover, by Lemma~\ref{lem3.6}, there exists
a threshold $\lambda^*=\Lambda_1>0$ such that for all $\lambda>\lambda^*$
the functional $I_\lambda$ admits a nontrivial critical point
$u\in W_{V}^{s,N/s}(\mathbb{R}^N)$. The critical point $u_\lambda$ is a mountain pass solution
of equation~\eqref{eq1}. Using a similar discussion as in Lemma \ref{lem3.5}, we can deduce that
$\|u_\lambda\|\rightarrow0$ as $\lambda\rightarrow\infty$.
Furthermore, Lemma \ref{lem3.1} shows that $u$ is nonnegative.
\end{proof}
\section{Proof of Theorem \ref{th2}}\label{sec4}
Throughout this section we always assume that the conditions in Theorem \ref{th2} hold. To prove Theorem \ref{th2}, we first state several basic results.
\begin{lemma}\label{lem4.1} There exist $\Lambda_2>0$ and $\delta_2>0$ such that for $0<\lambda<\Lambda_2$, there exists $\widetilde{\rho}_\lambda>0$ so that
$I_\lambda(u)\geq\delta_2>0$ for any $u\in W_V^{s,N/s}(\mathbb{R}^N)$, with $\|u\|=\widetilde{\rho}_\lambda$. Furthermore, $\widetilde{\rho}_\lambda$ can be chosen such that $\widetilde{\rho}_\lambda\rightarrow0$ as $\lambda\rightarrow0$.
\end{lemma}

\begin{proof}
By using  \eqref{c6}, \eqref{M1} and the H\"{o}lder inequality,  we obtain for all $u\in W_V^{s,N/s}(\mathbb{R}^N)$, with $\|u\|\leq 1$ small enough,
\begin{align*}
 I_\lambda(u)&\geq\frac{s\mathscr{M}(1)}{N}\|u\|^{\theta N/s}-\frac{s\mathscr{M}(1)}{N}\frac{(\lambda_1-\tau)}{\lambda_1}\|u\|^{\theta N/s}- C(N,s,\alpha)\|u\|^q-\lambda
\|h\|_{L^{\frac{N}{N-sp}}(\mathbb{R}^N)}\|u\|_{L^{N/s}(\mathbb{R}^N)}^p\nonumber\\
&= \frac{\tau s\mathscr{M}(1)}{\lambda_1 N}\|u\|^{\theta N/s}-C(N,s,\alpha)\|u\|^q-\lambda
 S_p^p
\|h\|_{L^{\frac{N}{N-sp}}(\mathbb{R}^N)}\|u\|^p.
\end{align*}
Hence,
\begin{align*}
I_\lambda(u)\geq\left(\frac{\tau s\mathscr{M}(1)}{\lambda_1 N}\|u\|^{\theta N/s-p}-C(N,s,\alpha)\|u\|^{q-p}-\lambda
 S_p^p
\|h\|_{L^{\frac{N}{N-sp}}(\mathbb{R}^N)}\right)\|u\|^p.
\end{align*}
Since $1<p<\theta N/s<q$, we can choose $\widetilde{\rho}_\lambda\in (0,1)$ such that $\frac{\tau s\mathscr{M}(1)}{\lambda_1 N}\widetilde{\rho}_\lambda^{\theta N/s-1}-C(N,s,\alpha)\widetilde{\rho}_\lambda^{q-1}>0$. Thus, $I_\lambda(u)\geq \delta_1:=
\left(\frac{\tau s\mathscr{M}(1)}{\lambda_1 N}\widetilde{\rho}_\lambda^{\theta N/s-p}-C(N,s,\alpha)\widetilde{\rho}_\lambda^{q-p}-\lambda S_p^p
\|h\|_{L^{\frac{N}{N-sp}}(\mathbb{R}^N)}\right)\widetilde{\rho}_\lambda^p>0$ for all $u\in W_V^{s,N/s}(\mathbb{R}^N)$, with $\|u\|=\widetilde{\rho}_\lambda$ and all $0<\lambda<\Lambda_2:=\left(\frac{\tau s\mathscr{M}(1)}{\lambda_1 N}\widetilde{\rho}_\lambda^{\theta N/s-p}-C(N,s,\alpha)\widetilde{\rho}_\lambda^{q-p}\right)/(S_p^p
\|h\|_{L^{\frac{N}{N-sp}}(\mathbb{R}^N)})$.
\end{proof}

\begin{lemma}\label{lem4.2}
There exists
$\Lambda_3>0$ such that for all $0<\lambda<\Lambda_3$, the functional
$I_\lambda$ satisfies the $(PS)_c$ condition for $c\leq0$.
\end{lemma}
\begin{proof}
Fix $c\leq0$ and assume that $\{u_n\}_n\subset
W_V^{s,N/s}(\mathbb{R}^N)$ satisfies
\begin{align*}
I_\lambda(u_n)\rightarrow c,\ \ I_\lambda^\prime(u_n)\rightarrow0
\ \ {\rm as}\ n\rightarrow\infty.
\end{align*}
If $d:=\inf_{n\geq1}\|u_n\|=0$, then up to a subsequence, we can get that $u_n\rightarrow0$ in $W_V^{s,N/s}(\mathbb{R}^N)$.

In the following, we assume that $d:=\inf_{n\geq1}\|u_n\|>0$.
Proceeding as in \eqref{eq3.9}, we can deduce
\begin{align*}
\left(\frac{s}{N\theta}-\frac{1}{\mu}\right)M(\|u_n\|^{N/s})\|u_n\|^{N/s}\leq
\lambda\left(\frac{1}{p}-\frac{1}{\mu}\right)S_p^p\|h\|_{L^{\frac{N}{N-sp}}(\mathbb{R}^N)}\|u_n\|^p+c+o(1)+o(1)\|u_n\|,
\end{align*}
which means that $\{u_n\}_n$ is bounded in $W_V^{s,N/s}(\mathbb{R}^N)$.
By $(M_1)$, we then get

\begin{align*}
\left[\left(\frac{s}{N\theta}-\frac{1}{\mu}\right)
\kappa(d)\|u_n\|^{N/s-p}-
\lambda\left(\frac{1}{p}-\frac{1}{\mu}\right)S_p^p\|h\|_{L^{\frac{N}{N-sp}}(\mathbb{R}^N)}\right]
\|u_n\|^p\leq o(1)+o(1)\|u_n\|.
\end{align*}
It follows that
\begin{align*}
\limsup_{n\rightarrow\infty}\|u_n\|\leq \left[\frac{\kappa(d)}{\left(\frac{s}{N\theta}-\frac{1}{\mu}\right)}
\lambda\left(\frac{1}{p}-\frac{1}{\mu}\right)S_p^p\|h\|_{L^{\frac{N}{N-sp}}(\mathbb{R}^N)}\right]^{\frac{s}{N-sp}}
\end{align*}
Set $$\Lambda_3=\frac{p(s\mu-N\theta)}{\kappa(d)N\theta(\mu-p)S_p^p\|h\|_{L^{\frac{N}{N-sp}}(\mathbb{R}^N)}}
\left(\frac{\alpha_{N,s}}{\alpha_0}\right)^{\frac{(N-s)(N\theta-s)}{Ns}}.$$
Then for all $0<\lambda<\Lambda_3$, we get
\begin{align*}
\limsup_{n\rightarrow\infty}
\|u_n\|<\left(\frac{\alpha_{N,s}}{\alpha_0}\right)^{\frac{N-s}{N}}.
\end{align*}
By using the same argument as in Lemma \ref{lem3.6}, we can prove that $I_\lambda$ satisfies the $(PS)_c$ condition for all $c\leq0$.
\end{proof}
\begin{proof}[\bf Proof of Theorem \ref{th2}]
Choosing a function $0\leq v\in W_V^{s,N/s}(\mathbb{R}^N)\setminus\{0\}$ with $\|v\|=1$ and $\int_{\mathbb{R}^N}h(x)v^pdx>0$,
we can deduce from $(f_2)$ that
\begin{align*}
I_\lambda(tv)\leq \left(\max_{0\leq \tau\leq 1}M(\tau)\right)\frac{st^{N/s}}{N}
-\lambda t^p\int_{\mathbb{R}^N}h(x)|v|^pdx
\end{align*}
for all $0\leq t\leq 1$.
Since $N/s>p$, it follows that $I_\lambda(tv)<0$ for $t\in(0,1)$ small enough. Thus,
\begin{align*}
c=\inf_{u\in \overline{B}_{\widetilde{\rho}_\lambda}}I_\lambda(u)<0\ \ {\rm and}\
\inf_{u\in\partial B_{\widetilde{\rho}_\lambda}}I_\lambda(u)>0,
\end{align*}
where $\widetilde{\rho}_\lambda>0$ is given by Lemma \ref{lem4.1} and
$B_{\widetilde{\rho}_\lambda}=\{u\in W_V^{s,N/s}(\mathbb{R}^N):\|u\|<\widetilde{\rho}_\lambda\}$. We can choose $\lambda$ small enough such that
\begin{align*}
\widetilde{\rho}_\lambda<\left(\frac{\alpha_{N,s}}{\alpha_0}\right)^{(N-s)/N}.
\end{align*}
Let $\varepsilon_n\rightarrow 0$ be such that
\begin{align}\label{eq4.11}
0<\varepsilon_n<\inf_{u\in\partial B_{\widetilde{\rho}_\lambda}}I_\lambda(u)-
\inf_{u\in \overline{B}_{\widetilde{\rho}_\lambda}}I_\lambda(u).
\end{align}
By Ekeland's variational principle, there exists $\{u_n\}\subset \overline{B}_{\widetilde{\rho}_\lambda}$ such that
\begin{align*}
c\leq I_\lambda(u_n)\leq c+\varepsilon_n
\end{align*}
and
\begin{align*}
I_\lambda(u_n)<I_\lambda (w)+\varepsilon\|u_n-w\|,\ \ \forall w\in \overline{B}_{\widetilde{\rho}_\lambda}, w\neq u_n.
\end{align*}
Then, from \eqref{eq4.11} and the definition of $c$, we get
\begin{align*}
I_\lambda(u_n)\leq c+\varepsilon_n=\inf_{u\in\overline{B}_{\widetilde{\rho}_\lambda}}I_\lambda(u)+\varepsilon_n<\inf_{u\in\partial B_{\widetilde{\rho}_\lambda}}I_\lambda(u)
\end{align*}
and thus $\{u_n\}_n\subset B_{\widetilde{\rho}_\lambda}$.

Consider the sequence $v_n=u_n+t\varphi\subset B_{\widetilde{\rho}_\lambda}$ for all $\varphi\in B_1$ and $t>0$ small enough. Then it follows that
\begin{align*}
\frac{I_\lambda(u_n+t\varphi)-I_\lambda(u_n)}{t}\geq -\varepsilon_n\|\varphi\|.
\end{align*}
Passing to the limit as $t\rightarrow 0$, we deduce
\begin{align*}
\langle I^\prime_\lambda(u_n),\varphi\rangle\geq -\varepsilon_n\|\varphi\|,\ \ \forall \varphi\in B_1.
\end{align*}
Replacing $\varphi$ with $-\varphi$, we have
\begin{align*}
\langle I^\prime_\lambda(u_n),-\varphi\rangle\geq -\varepsilon_n\|\varphi\|,\ \ \forall \varphi\in B_1.
\end{align*}
Then
\begin{align*}
|\langle I^\prime_\lambda(u_n),\varphi\rangle|\leq \varepsilon_n,\ \ \forall \varphi\in B_1
\end{align*}
and thus
\begin{align*}
\|I_\lambda^\prime(u_n)\|\rightarrow0\ \ {\rm as}\ n\rightarrow\infty.
\end{align*}
Therefore there exists a sequence $\{u_n\}_n\subset B_{\rho_\lambda}$ such that
$I_\lambda(u_n)\rightarrow c\leq0$ and $I^\prime(u_n)\rightarrow 0$, as $n\rightarrow\infty$. Observing that
\begin{align*}
\|u_n\|\leq \widetilde{\rho}_\lambda<\left(\frac{\alpha_{N,s}}{\alpha_0}\right)^{(N-s)/N},
\end{align*}
by Lemma \ref{lem4.2}, there exists $0<\lambda_*\leq\min\{\Lambda_2,\Lambda_3\}$ such that for all $\lambda\in(0,\lambda_*)$, $\{u_n\}_n$ has a convergent subsequence, still denoted by $\{u_n\}_n$, such that $u_n\rightarrow u_\lambda$ in $W_V^{s,N/s}(\mathbb{R}^N)$. Thus, $u_\lambda$ is a nontrivial nonnegative solution with $I_\lambda(u_\lambda)<0$.  Moreover, $\|u_\lambda\|\leq \widetilde{\rho}_\lambda\rightarrow 0$ as $\lambda\rightarrow0$. Hence, the proof is complete.
\end{proof}

\section{Proof of Theorem \ref{th3}}\label{sec5}
In this section, we discuss the multiplicity of solutions for \eqref{eq1}.  To this end, we first recall some basic notions about the Krasnoselskii genus.

Let $X$ be a Banach space and $A$ a subset of $X$. $A$ is said to be symmetric if $u\in A$ implies
$-u\in A$. Let us denote by $\Xi$ the family of closed symmetric subsets $A\subset X\setminus\{0\}$.
\begin{definition}
Let $A\in \Xi$. The Krasnoselskii genus $\gamma(A)$ of $A$ is defined as
the least positive integer $k$ such that there is an odd mapping $\varphi\in C(A,\mathbb{R}^k)$ such
that $\varphi(x)\neq 0$ for all $x\in A$.  If $k$ does not exist, we set $\gamma(A)=\infty$. Furthermore, by
definition, $\gamma(\emptyset)=0$.
\end{definition}
In the sequel we list some properties of the genus that will be
used later. For more details on this subject, we refer to \cite{Rabin}.
\begin{proposition}\label{Prop5.1}
Let $A,B$ be sets in $\Xi$.
\begin{itemize}
\item[$(1)$] If there exists an odd map $\varphi\in C(A,B)$, then $\gamma(A)\leq \gamma(B)$.
\item[$(2)$] If $A\in \Gamma$ and $\gamma(A)\geq 2$, then $A$ has infinitely
many points.
\item[$(3)$] If $A\subset B$, then $\gamma(A)\leq \gamma(B)$.
\item[$(4)$] $\gamma(A\bigcup B)\leq \gamma(A)+\gamma(B)$.
\item[$(5)$] If $\mathbb{S}$ is a sphere centered at the origin in $\mathbb{R}^k$, then
$\gamma(\mathbb{S})=k$.
\item[$(6)$] If $A$ is compact, then $\gamma(A)<\infty$ and
there exists $\delta>0$ such that $N_\delta(A)\in\Xi$ and
$\gamma(N_\delta(A))=\gamma(A)$, where $N_\delta(A)=
\{x\in X: \|x-A\|\leq \delta\}.$
 \end{itemize}
\end{proposition}

Define
\begin{align*}
& \mathcal{I}_\lambda(u)=\frac{s}{N}\mathscr{M}(\|u\|^{N/s})
-\int_{\mathbb{R}^N}F(x,u)dx-
\lambda\int_{\mathbb{R}^N}h(x)|u|^{p}dx,
\end{align*}
for all $u\in W_V^{s,N/s}(\mathbb{R}^N)$. Clearly, under assumption $(f_1)$, $\mathcal{I}_\lambda\in C^1(W_V^{s,N/s}(\mathbb{R}^N),\mathbb{R})$ and the critical points are the weak solutions of \eqref{eq1}.

Following the ideas of \cite{AA} (see also \cite{HZ}), we construct a truncated functional $J_\lambda$ such that critical points $u$ of $J_\lambda$ with $J_\lambda(u)<0$ are also critical points of $\mathcal{I}_\lambda$. Since the system \eqref{eq1} contains a nonlocal coefficient $M(\|u\|^{N/s})$ and the operator $(-\Delta)_{N/s}^s$ is nonlocal, our task is complicated. To overcome these difficulties in the building of $J_\lambda$, we split the discussion into two cases $\|u\|\leq 1$ and $\|u\|>1$.
\smallskip

\noindent{\it Case 1}: $\|u\|\leq 1$. By $(f_1)$ and $(f_3^\prime)$, we obtain for any $\tau\in(0,\lambda_1)$ and $q>N\theta/s$ there exists $C>0$ such that
 \begin{align*}
F(x,u)\leq (a+b)\frac{s}{N}(\lambda_1-\tau)|u|^{N\theta/s}+C|u|^q\Phi_\alpha(u)
\end{align*}
for all $u\in\mathbb{R}$ and $x\in\mathbb{R}^N$. Furthermore, from the definition of $\mathcal{I}$, there exist $\tau\in(0,\lambda_1)$ and $C(N, s, \alpha)>0$ such that
\begin{align*}
\mathcal{I}_\lambda(u)\geq\frac{\tau s(a+b)}{\lambda_1 N}\|u\|^{\theta N/s}-C(N,s,\alpha)\|u\|^q-\frac{\lambda
S_{N/s}^p}{p}
\|h\|_{L^{\frac{N}{N-sp}}(\mathbb{R}^N)}\|u\|^p
\end{align*}
for all $u\in W_V^{s,N/s}(\mathbb{R}^N)$ with $\|u\|\leq 1$, where $S_{N/s}$ is the embedding constant from $W_V^{s,N/s}(\mathbb{R}^N)$ to $L^{N/s}(\mathbb{R}^N)$.
Define $$G_\lambda(t)=\frac{\tau s(a+b)}{\lambda_1 N}t^{\theta N/s}-C(N,s,\alpha)t^q-\frac{\lambda
 S_{N/s}^p}{p}
\|h\|_{L^{\frac{N}{N-sp}}(\mathbb{R}^N)}t^p,$$
for all $t\geq0$. Then
\begin{align}\label{eq5.1}
\mathcal{I}_\lambda(u)\geq G_\lambda(\|u\|)
\end{align}
for all $u\in W_V^{s,N/s}(\mathbb{R}^N)$ with $\|u\|\leq 1$.
 Since $p<\theta N/s<q$, there exists $\lambda_{**}\in (0,\lambda_*]$ small enough such that $G_\lambda$ attains its positive maximum for $\lambda\in (0,\lambda_{**})$. Here $\lambda_*>0$ is given by Theorem \ref{th2}. Denote by $0<T_0(\lambda)<T_1(\lambda)$ the unique two positive roots of $G_\lambda(t)=0$. Indeed, to get the solutions of $G_\lambda(t)=0$ for all $t>0$, one can consider $\widetilde{G}_\lambda$ defined as
\begin{align*}
\widetilde{G}_\lambda(t)=\left[\frac{\tau s(a+b)}{\lambda_1 N}t^{\theta N/s-p}-C(N,s,\alpha)t^{q-p}-\frac{\lambda
S_{N/s}^p}{p}
\|h\|_{L^{\frac{N}{N-sp}}(\mathbb{R}^N)}\right]t^p
\end{align*}
for all $t\geq0$.

Actually, $T_0(\lambda)$ has the following property.
\begin{lemma}\label{lem5.1}
$\lim\limits_{\lambda\rightarrow 0^+}T_0(\lambda)=0$.
\end{lemma}
\begin{proof}
By $G_\lambda(T_0(\lambda))$ and $G_\lambda^\prime(T_0(\lambda))>0$, we have
\begin{align}\label{eq5.2}
\frac{\tau s(a+b)}{\lambda_1 N}T_0(\lambda)^{N\theta /s}=\frac{\lambda
S_{N/s}^p}{p}
\|h\|_{L^{\frac{N}{N-sp}(\mathbb{R}^N)}}
T_0(\lambda)^{p}+C(N,s,\alpha)
T_0(\lambda)^{q}
\end{align}
and
\begin{align}\label{eq5.3}
\frac{\tau \theta(a+b)}{\lambda_1 }T_0(\lambda)^{N\theta /s-1}>\lambda
 S_{N/s}^p
\|h\|_{L^{\frac{N}{N-sp}}(\mathbb{R}^N)}
T_0(\lambda)^{p-1}+qC(N,s,\alpha)
T_0(\lambda)^{q-1}.
\end{align}
Combining \eqref{eq5.2} and \eqref{eq5.3}, we get
\begin{align*}
T_0(\lambda)\leq \left(\frac{(a+b)(N\theta-ps)\tau}{\lambda_1N(q-p)C(N,s,\alpha)}\right)^{\frac{s}{sq-N\theta}},
\end{align*}
which means that $T_0(\lambda)$ is uniformly bounded with respect to $\lambda$.
Fix any sequence $\{\lambda_k\}_k\subset (0,\infty)$, with $\lambda_k\rightarrow 0$ as $k\rightarrow\infty$.
Assume that $T_0(\lambda_k)\rightarrow T_0$ as $k\rightarrow \infty$. Then by \eqref{eq5.2} and \eqref{eq5.3}, we have
\begin{align}\label{eq5.4}
\frac{\tau s(a+b)}{\lambda_1 N}T_0^{N\theta /s}=C(N,s,\alpha)
T_0^{q}
\end{align}
and
\begin{align}\label{eq5.5}
\frac{\tau \theta(a+b)}{\lambda_1 }T_0^{N\theta /s-1}\geq qC(N,s,\alpha)
T_0^{q-1}.
\end{align}
It follows from \eqref{eq5.4} and \eqref{eq5.5} that
\begin{align*}
\left(\frac{s}{N}-\frac{\theta}{q}\right)\frac{\tau(a+b)}{\lambda_1}T_0^{N\theta /s}\leq0,
\end{align*}
which implies that $T_0=0$, thanks to $q>N\theta /s$. The arbitrary choice of $\{\lambda_k\}_k$ yields that
$\lim\limits_{\lambda\rightarrow 0^+}T_0(\lambda)=0$.
This completes the proof.
\end{proof}
By Lemma \ref{lem5.1}, we can assume that $T_0(\lambda)<1$ for small enough $\lambda$. Thus, $T_0(\lambda)<\min\{T_1(\lambda),1\}$. Take
$\Psi\in C_0^\infty([0,\infty))$, $0\leq \Psi(\tau)\leq 1$ for all $\tau\geq 0$ and
\begin{align*}
\Psi(t)=
\begin{cases}
1,\ \  &\mbox{if}\ t\in [0,T_0(\lambda)],\\
0,\ \ &\mbox{if}\ t\in [\min\{T_1(\lambda),1\},\infty).
\end{cases}
\end{align*}
Then we define the functional
\begin{align*}
J_\lambda(u)=\frac{s}{N}\mathscr{M}(\|u\|^{N/s})-\frac{\lambda}{p}\int_{\mathbb{R}^N}h(x)|u|^{p}dx-\Psi(\|u\|)\int_{\mathbb{R}^N}F(x,u)dx.
\end{align*}
One can easily verify that $J_\lambda\in C^1(W_V^{s,N/s}(\mathbb{R}^N),\mathbb{R})$ and $J_\lambda(u)\geq H_\lambda(\|u\|)$ for all $u\in W_V^{s,N/s}(\mathbb{R}^N)$ with $\|u\|<1$, where
\begin{align}
H_\lambda(t):=\frac{\tau s(a+b)}{\lambda_1 N}t^{\theta N/s}-C(N,s,\alpha)\Psi(t)t^q-\frac{\lambda
S_{N/s}^p}{p}
\|h\|_{L^{\frac{N}{N-sp}}(\mathbb{R}^N)}t^p.
\end{align}

 Clearly, $H_\lambda(t)\geq G_\lambda(t)\geq 0$ for all $t\in (T_0(\lambda),\min\{T_1(\lambda),1\}]$. By the definitions of $\mathcal{I}_\lambda$ and $J_\lambda$, we know that $\mathcal{I}_\lambda(u)=J_\lambda(u)$ for all $\|u\|\leq T_0(\lambda)<\min\{T_1(\lambda),1\}$. Let $u$ be a critical point of $J_\lambda$ with $J_\lambda(u)<0$. If $\|u\|<T_0(\lambda)$, then $u$ is also a critical point of $I_\lambda$.  To show that $\|u\|<T_0(\lambda)$ it is important to ensure that $J_\lambda(u)\geq 0$  when $\|u\|\geq 1$.
\smallskip

\noindent{\it Case 2}: $\|u\|> 1$.
Note that in this case we always have $\Psi(\|u\|)=0$. Hence, for all $\|u\|>1$, we obtain by  $(M_3)$ that
\begin{align*}
J_\lambda(u)&=\frac{s}{N}\mathscr{M}(\|u\|^{N/s})-\frac{\lambda}{p}\int_{\mathbb{R}^N}h(x)|u|^{p}dx\\
&\geq \frac{s(a+b)}{N} \|u\|^{N/s}-\frac{\lambda}{p}S_{N/s}^p\|h\|_{L^{\frac{N}{N-sp}}(\mathbb{R}^N)}\|u\|^p\\
&=\widetilde{g}(\|u\|),
\end{align*}
where $\widetilde{g}:[0,\infty)\rightarrow\mathbb{R}$ is defined by
\begin{align*}
\widetilde{g}(t)=\frac{s(a+b)}{N} t^{N/s}-\frac{\lambda}{p}S_{N/s}^p\|h\|_{L^{\frac{N}{N-sp}}(\mathbb{R}^N)}t^p.
\end{align*}
It is easy to check that $\widetilde{g}$ has a global minimum point at $t_\lambda=\left(\frac{1}{(a+b)}\lambda S_{N/s}^p\|h\|_{L^{\frac{N}{N-sp}}(\mathbb{R}^N)}\right)^{\frac{s}{N-sp}}$ and
\begin{align*}
\widetilde{g}(t_\lambda)=\left(\frac{1}{(a+b)^{\frac{sp}{N}}}\lambda S_{N/s}^p\|h\|_{L^{\frac{N}{N-sp}}(\mathbb{R}^N)}\right)^{\frac{N}{N-sp}}\left(\frac{s}{N}-\frac{1}{p}\right)<0,
\end{align*}
being $p<N/s$. Observe that $\widetilde{g}(t)\geq 0$ if and only if $t\geq \left(\frac{\lambda N S_{N/s}^p}
{sp(a+b)}\|h\|_{L^{\frac{N}{N-sp}}(\mathbb{R}^N)}\right)^{\frac{s}{N-sp}}:=t_0 $. Thus, to ensure that $J_\lambda(u)\geq 0$ for all $\|u\|\geq 1$, we let $t_0<1$, that is, $\lambda< \frac{sp(a+b)}{N S_{N/s}^p \|h\|_{L^{\frac{N}{N-sp}}(\mathbb{R}^N)}}:=\lambda_0$.
Hence,
we take $\lambda_{**}\leq  \lambda_0$. Then for each $\lambda\in (0,\lambda_{**})$ we have $J_\lambda(u)\geq0$ for any $\|u\|> 1$.

\begin{lemma}\label{lem5.2}
Let $\lambda\in (0,\lambda_{**})$. If $J_\lambda(u)<0$, then $\|u\|<T_0(\lambda)$ and $J_\lambda(v)=\mathcal{I}_\lambda(v)$ for all $v$ in a small enough neighbourhood of $u$. Moreover, $J_\lambda$ satisfies a local $(PS)_c$ condition for all $c<0$.
\end{lemma}
\begin{proof}
Since $\lambda\in (0,\lambda_{**})$, $J_\lambda(u)\geq 0$ for all $\|u\|\geq 1$. Thus, if $J_\lambda(u)<0$ we have $\|u\|<1$
 and consequently $G_\lambda(\|u\|)\leq J_\lambda(u)<0$. Therefore $\|u\|<T_0(\lambda)$ and $J_\lambda(u)=\mathcal{I}_\lambda(u)$. Moreover, $J_\lambda(v)=\mathcal{I}_\lambda(v)$ for all $v$ satisfying $\|v-u\|<T_0(\lambda)-\|u\|$.
 Let $\{u_n\}_n$ be a sequence such that
 $J_\lambda(u_n)\rightarrow c<0$ and $J_\lambda^\prime(u_n)\rightarrow 0$. Then for $n$ sufficiently large, we have
 $\mathcal{I}_\lambda(u_n)=J_\lambda(u_n)\rightarrow c<0$ and $\mathcal{I}^\prime_\lambda(u_n)=J_\lambda^\prime (u_n)\rightarrow 0$.
 Note that $J_\lambda$ is coercive in $W_V^{s,N/s}(\mathbb{R}^N)$. Thus, $\{u_n\}_n$ is bounded in $W_V^{s,N/s}(\mathbb{R}^N)$. By using a similar discussion as Lemma \ref{lem4.2}, up to a subsequence, $\{u_n\}_n$  is strongly convergent in  $W_V^{s,N/s}(\mathbb{R}^N)$.
\end{proof}

\begin{remark}\label{rem5.1}
Set $K_c=\{u\in W_V^{s,N/s}(\mathbb{R}^N): J_\lambda^\prime(u)=0,
J_\lambda(u)=c\}$. If $\lambda\in (0,\lambda_{**})$ and $c<0$, it follows
from Lemma \ref{lem5.2} that $K_c$ is compact.
\end{remark}
Next, we will construct an appropriate mini-max sequence of negative critical
values for the functional $J_\lambda$. For $\epsilon>0$, we define
\begin{align*}
J_\lambda^{-\epsilon}=\{u\in W_V^{s,N/s}(\mathbb{R}^N):J_\lambda(u)\leq
-\epsilon\}.
\end{align*}
\begin{lemma}\label{lem5.3}
For any fixed $k\in \mathbb{N}$ there exists $\epsilon_k>0$ such that
$$\gamma(J_\lambda^{-\epsilon_k})\geq k.$$
\end{lemma}
\begin{proof}
Denote by $E_k$ an $k$-dimensional subspace of $W_V^{s,N/s}(\mathbb{R}^N)$.
For any $u\in E_k$, $u\neq 0$, set $u=r_kv$ with $v\in E_k$, $\|v\|=1$ and
$r_k=\|u\|$. By the assumption on $h$, we know that $
(\int_{\mathbb{R}^N}h(x)|v|^pdx)^{1/p}$ is a norm of $E_k$.
Since all norms are equivalent in a finite-dimensional
Banach space, for each $v\in E_k$ with $\|v\|=1$, there
exists $C_k>0$ such that
\begin{align*}
\int_{\mathbb{R}^N}h(x)|v|^pdx\geq C_k.
\end{align*}
Thus, for $r_k\in(0,T_0(\lambda))$, we have
\begin{align*}
J_\lambda(u)=\mathcal{I}_\lambda(u)&=\frac{s}{N}\mathscr{M}(\|u\|^{N/s})-\frac{\lambda}{p}\int_{\mathbb{R}^3}h(x)|u|^{p}dx-\int_{\mathbb{R}^N}
F(x,u)dx\\
&\leq \frac{s(a+b)}{N}r_k^{N\theta/s}-
\frac{\lambda}{p}C_kr_k^p
.
\end{align*}
Since $p<N\theta/s$, we can choose $r_k\in (0,T_0(\lambda))$ so small that
$J_\lambda(u)\leq -\epsilon_k<0$. Set $S_{r_k}=\{u\in W_V^{s,N/s}(\mathbb{R}^N):\|u\|=r_k\}$. Then
$S_{r_k}\bigcap E_k\subset J_\lambda^{-\epsilon_k}$. Hence, it follows from Proposition \ref{Prop5.1} that
$\gamma(J_\lambda^{-\epsilon_k})\geq \gamma(S_{r_k}\bigcap E_k)=k$.
\end{proof}

Set $\Xi_k=\{A\in \Xi:\gamma(A)\geq k\} $ and let
\begin{align}\label{CV}
c_k:=\inf_{A\in\Xi_k}\sup_{u\in A}J_\lambda(u).
\end{align}
Then,
\begin{align*}
-\infty<c_k\leq -\epsilon_k<0, \ \forall k\in\mathbb{N},
\end{align*}
since $J_\lambda^{-\epsilon_k}\in \Xi_k$ and $J_\lambda$ is bounded
from below.
By \eqref{CV}, we have $c_k<0$.
Since $J_\lambda$ satisfies
$(PS)_c$ condition by Lemma \ref{lem5.2}, it follows by a standard
argument that
all $c_k$ are critical values
of $J_\lambda$.

\begin{lemma}\label{lem5.4}
Let $\lambda\in (0,\lambda_{**})$.  If $c=c_k=c_{k+1}=\cdots=c_{k+m}$ for some $m\in\mathbb{N}$, then
$\gamma(K_c)\geq m+1$.
\end{lemma}
\begin{proof}
Arguing by contradiction, we assume that $\gamma(K_c)\leq m$.
By Remark \ref{rem5.1}, we know that $K_{c}$ is compact
and $K_{c}\in\Xi$.
 It follows
from Proposition \ref{Prop5.1} that
 there exists
$\delta>0$ such that
\begin{align*}
\gamma(K_{c})=\gamma(N_\delta(K_{c_\infty}))\leq m.
\end{align*}
From the deformation lemma (see \cite[Theorem A.4]{Rabin}),
there exist $0<\epsilon<-c$, and an odd homeomorphism
$\eta:W_V^{s,N/s}(\mathbb{R}^N)\rightarrow W_V^{s,N/s}(\mathbb{R}^N)$
such that
\begin{align}\label{eq5.8}
\eta(J_\lambda^{c+\epsilon}\setminus
N_\delta(K_{c})
\subset J_\lambda^{c-\epsilon}.
\end{align}
On the other hand, by the definition of $c=c_{k+m}$,
there exists $A\in \Xi_{k+m}$ such that
 $\sup_{u\in A}J_\lambda(u)<c+\epsilon$, which means that
 \begin{align*}
 A\subset J_\lambda^{c+\epsilon}.
 \end{align*}
 It follows from Proposition \ref{Prop5.1} that
 \begin{align*}
 \gamma(\overline{A\setminus N_\delta(K_{c})})\geq \gamma (A)
 -\gamma(N_\delta(K_{c}))\geq k
 \end{align*}
 and
 \begin{align*}
 \gamma(\overline{\eta(A\setminus N_\delta(K_{c}))})\geq k.
 \end{align*}
Thus,
\begin{align*}
\overline{\eta(A\setminus N_\delta(K_{c}))}\in \Xi_k,
\end{align*}
which contradicts \eqref{eq5.8}. This completes the proof.
\end{proof}
\begin{proof}[\bf Proof of Theorem \ref{th3}]
Let $\lambda\in(0,\lambda_{**})$.
If $-\infty<c_1<c_2<\cdots<c_k<\cdots<0$, since $c_k$ are critical values of
$J_\lambda$,  we obtain infinitely many critical points of $J_\lambda$.
From Lemma \ref{lem5.2}, $\mathcal{I}_\lambda=J_\lambda$ if $J_\lambda<0$. Hence system \eqref{eq1}
 has infinitely many solutions.

 If there exist $c_k=c_{k+m}$, then
 $c=c_k=c_{k+1}=\cdots=c_{k+m}$. By Lemma \ref{lem5.4}
 ,  we have $\gamma(K_c)\geq m+1\geq 2$.  From (2) of Proposition \ref{Prop5.1},
$K_c$ has infinitely many points. Thus, system \eqref{eq1} has infinitely many solutions.
The proof is now complete.
\end{proof}

It is natural to consider the existence of infinitely many solutions for problem \ref{eq1} in the case $1<p<N \theta/s$. For this, we replace $M(t)=a+b\theta t^{\theta-1}$ with $M(t)=b^{\theta-1}$. Hence, by employing the same approach as Theorem \ref{th3},  we can get the following result.

\begin{corollary}\label{col}
Assume that $V$ satisfies $(V_1)$--$(V_2)$, and $f$ satisfies $(f_1)$--$(f_3)$. If $1<p<N\theta/s$ and $0\leq h\in L^{\frac{N\theta}{N\theta-sp}}(\mathbb{R}^N)$,
then there exists
$\lambda_{**}\in(0, \lambda_*]$ such that for all $0<\lambda<\lambda_{**}$, problem \eqref{eq1}
has infinitely many solutions in $ W_V^{s,N/s}(\mathbb{R}^N)$.
\end{corollary}

\section{Extensions to a nonlocal integro--differential operator}\label{sec6}
In this section, we show that Theorems~\ref{th1}--\ref{th2} remain valid when $(-\Delta )_{N/s}^s$ in \eqref{eq1} is replaced by a nonlocal integro--differential operator $\mathcal{L}_{\mathcal{K}}$, defined by
\begin{align*}
\mathcal{L}_{\mathcal{K}}(\varphi)=2\lim_{\varepsilon\rightarrow 0^+}\int_{\mathbb{R}^N\setminus B_\varepsilon(x)} |\varphi(x)-\varphi(y)|^{\frac{N}{s}-2}(\varphi(x)-\varphi(y))\mathcal{K}(x-y)dxdy,
\end{align*}
along any function $\varphi\in C_0^\infty(\mathbb{R}^N)$, where {\em the singular kernel $\mathcal{K}:\mathbb{R}^N\setminus\{0\}\rightarrow \mathbb R^+$ satisfies the following properties}:
\begin{itemize}
\item[$(k_1)$]
$m \mathcal{K}\in L^1(\mathbb{R}^N)$, {\em where} $m(x)=\min\{1,|x|^{N/s}\}$;
\item[$(k_2)$] {\em there exists $\mathcal{K}_0>0$ such that $\mathcal{K}(x)\geq \mathcal{K}_0|x|^{-2N}$ for all} $x\in\mathbb{R}^N\setminus\{0\}$.
\end{itemize}
Obviously, $\mathcal{L}_{\mathcal{K}}$ reduces to the fractional $N/s$--Laplacian $(-\Delta)^s_{N/s}$ when $\mathcal{K}(x)=|x|^{-2N}$.

Let us denote by $W_{V,\mathcal{K}}^{s,N/s}(\mathbb{R}^N)$ the completion of $C_0^\infty(\mathbb{R}^N)$ with respect to the norm
\begin{align*}
\|u\|_{V,\mathcal{K}}=\big([u]_{s,\mathcal{K}}^{N/s}
+\|u\|_{N/s,V}^{N/s}\big)^{s/N},\quad
[u]_{s,\mathcal{K}}=\left(\iint_{\mathbb{R}^{2N}}|u(x)-u(y)|^{N/s}{\mathcal{K}(x-y)}dxdy\right)^{s/N},
\end{align*}
here we apply $(k_1)$. Clearly, the embedding $W_{V,\mathcal{K}}^{s,N/s}(\mathbb{R}^N)\hookrightarrow W_{V}^{s,N/s}(\mathbb{R}^N)$ is continuous, being
\begin{align*}
[u]_{s,N/s}\leq\mathcal{ K}_0^{-1/p}[u]_{s, {\mathcal{K}}}\quad\mbox{for all }u\in W_{V,\mathcal{K}}^{s,N/s}(\mathbb{R}^N),
\end{align*}
by $(k_2)$. Hence Theorem~\ref{th2.1} remains valid and  the embedding $W_{V,\mathcal{K}}^{s,N/s}(\mathbb{R}^N)\hookrightarrow\hookrightarrow L^{\nu}(\mathbb{R}^N)$ is compact for all $\nu\geq N/s$ by virtue of $(V_1)$ and $(V_2)$.

A (weak) {\em solution} of
\begin{equation}\label{p2}
M(\|u\|_{V,\mathcal{K}}^{N/s})[\mathcal{L}_{\mathcal{K}}(u)+V(x)|u|^{N/s-2}u]= f(x,u)+\lambda h(x)|u|^{p-2}u
\ \ {\rm in}\ \mathbb{R}^N
\end{equation}
is a function $u\in W_{V,\mathcal{K}}^{s,N/s}(\mathbb{R}^N)$ such that
$$\begin{gathered}
M(\|u\|_{V,\mathcal{K}}^{N/s})\left(\langle u,\varphi\rangle_{s,\mathcal{K}}+\int_{\mathbb{R}^N}V|u|^{\frac{N}{s}-2}u\varphi dx\right)
=\int_{\mathbb{R}^N}f(x,u)\varphi dx
+\lambda\int_{\mathbb{R}^N}h(x)|u|^{p-2}u\varphi dx,\\
\langle u,\varphi\rangle_{s,\mathcal{K}}=\iint_{\mathbb{R}^{2N}}|u(x)-u(y)|^{\frac{N}{s}-2}(u(x)-u(y))(\varphi(x)-\varphi(y))\mathcal{K}(x-y)dxdy,
\end{gathered}$$
for all $\varphi\in W_{V,\mathcal{K}}^{s,N/s}(\mathbb{R}^N)$.

Here we  point out that it is not restrictive to assume $\mathcal{K}$ to be even{,} as in \cite{AFP}, since the odd part of {$\mathcal{K}$} does not give contribution in the integral of the left hand side. Indeed, we can write $\mathcal{K}(x)=\mathcal{K}_e(x)+
\mathcal{K}_o(x)$ for all $x\in\mathbb{R}^N\setminus\{0\}$, where
\begin{align*}
\mathcal{K}_e(x)=\frac{\mathcal{K}(x)+\mathcal{K}(-x)}{2}\ \ {\rm and}\
\ \mathcal{K}_o(x)=\frac{\mathcal{K}(x)-\mathcal{K}(-x)}{2}.
\end{align*}
Then by a direct calculation, one can get that
\begin{align*}
\langle u,\varphi\rangle_{s,\mathcal{K}}
=\iint_{\mathbb{R}^{2N}}|u(x)-u(y)|^{p-2}(u(x)-u(y))
(\varphi(x)-\varphi(y))\mathcal{K}_e(x-y)dxdy,
\end{align*}
for all $u$ and $\varphi \in W_{V,\mathcal{K}}^{s,N/s}(\mathbb{R}^N)$. Thus, it is not restrictive to assume that $\mathcal{K}$ is even.

The nontrivial solutions of~\eqref{p2} correspond to the critical points of the energy functional $ I_{\lambda,\mathcal{K}}:W_{V,\mathcal{K}}^{s,N/s}(\mathbb{R}^N)\rightarrow\mathbb{R}$, defined by
\begin{align*}
 I_{\lambda,\mathcal{K}}(u)=\frac{s}{N}\mathscr{M}(\|u\|_{V,\mathcal{K}}^{N/s})
-\int_{\mathbb{R}^N}F(x,u)dx-
\frac{\lambda}{p}
\int_{\mathbb{R}^{N}}h|u^+|^{p}dxdy
\end{align*}
for all $u\in W_{V,\mathcal{K}}^{s,N/s}(\mathbb{R}^N)$. Now we are able to prove the following results for problem~\eqref{p2}
by employing the parallel approach as in Theorems \ref{th1}--\ref{th3}.

\begin{theorem}
Assume that $V$ satisfies $(V_1)$--$(V_2)$,  $f$ satisfies $(f_1)$--$(f_2)$
and $M$ fulfills $(M_1)$--$(M_2)$. If $0\leq h\in L^\infty(\mathbb{R}^N)$ and $N\theta/s<p<\infty$, then there exists
$\widetilde{\lambda}^*>0$ such that for all $\lambda>\widetilde{\lambda}^*$ problem \eqref{p2}
admits
a nontrivial nonnegative mountain pass solution
$u_{\lambda}\in W_{V,\mathcal{K}}^{s,N/s}(\mathbb{R}^N)$.
Moreover,
\begin{equation*}
\lim_{\lambda\rightarrow\infty}\|u_\lambda\|_{V,\mathcal{K}}=0.\end{equation*}
\end{theorem}

\begin{theorem}
Assume that $V$ satisfies $(V_1)$--$(V_2)$, $f$ satisfies $(f_1)$--$(f_3)$,
and $M$ fulfills $(M_1)$--$(M_2)$.  If $1<p<N/s$ and $0\leq h\in L^{\frac{N}{N-sp}}(\mathbb{R}^N)$, then there exists $\widetilde{\lambda}_{*}>0$ such that for all
$\lambda\in (0, \widetilde{\lambda}_{*})$ problem~\eqref{p2}
admits a nontrivial nonnegative solution $u_{\lambda}\in W_{V,\mathcal{K}}^{s,N/s}(\mathbb{R}^N)$. Moreover,
\begin{equation*}
\lim_{\lambda\rightarrow0}\|u_\lambda\|_{V,\mathcal{K}}=0.
\end{equation*}
\end{theorem}
Moreover, we can prove the multiplicity of solutions.
\begin{theorem}
Assume that $V$ satisfies $(V_1)$--$(V_2)$, $f$ satisfies $(f_1)$--$(f_3)$,
and $M(t)=a+b\theta t^{\theta-1}$ for all $t\geq0$, with $a,b\geq0, a+b>0$ and $\theta>1$.  If $1<p<N\theta/s$ and $0\leq h\in L^{\frac{N\theta}{N\theta-sp}}(\mathbb{R}^N)$, then there exists $\widetilde{\lambda}_{**}\in (0,\widetilde{\lambda}_*]$ such that for all
$\lambda\in (0,\widetilde{\lambda}_{**})$ problem~\eqref{p2}
has infinitely many solutions in $W_{V,\mathcal{K}}^{s,N/s}(\mathbb{R}^N)$.
\end{theorem}

\noindent
{\bf Acknowledgements}.
The authors would like to express their deep gratitude to  anonymous referees for their valuable suggestions and useful comments.
Mingqi Xiang was supported by the National Natural Science Foundation of China
(No. 11601515) and the Fundamental Research Funds for the Central Universities (No. 3122017080).
Binlin Zhang  was supported by the National Natural Science Foundation of China
(No. 11871199).
Du\v{s}an Repov\v{s}
 was supported by the Slovenian Research Agency grants  P1-0292, J1-8131, J1-7025, N1-0064, and N1-0083.

\end{document}